\documentclass[12pt]{amsart}
\usepackage{amscd,amsmath,amssymb,amsfonts}
\usepackage{dsfont,enumerate}
\usepackage[cmtip, all]{xy}
\usepackage[top=30mm,right=30mm,bottom=30mm,left=20mm]{geometry}
\usepackage{graphicx,amssymb,amsmath,amsfonts,amsthm,amscd,hyperref,textcomp,relsize,colortbl}
\usepackage{tikz-cd}
\usetikzlibrary{babel}

\theoremstyle{plain}
\newtheorem{thm}{Theorem}
\newtheorem{lem}[thm]{Lemma}
\newtheorem{cor}[thm]{Corollary}
\newtheorem{prop}[thm]{Proposition}

\theoremstyle{definition}

\newtheorem{rmk}[thm]{Remark}

\numberwithin{thm}{section}
\numberwithin{equation}{thm}

\newcommand{\al}{\alpha}

\newcommand{\rank}{{\rm rank}}

\newcommand{\End}{{\rm End}}

\newcommand{\Gal}{{\rm Gal}}

\newcommand{\sB}{{\mathcal B}}

\newcommand{\sG}{{\mathcal G}}
\newcommand{\sH}{{\mathcal H}}

\newcommand{\sK}{{\mathcal K}}
\newcommand{\sL}{{\mathcal L}}

\newcommand{\A}{{\mathbb A}}
\newcommand{\B}{{\mathbb B}}
\newcommand{\C}{{\mathbb C}}

\newcommand{\F}{{\mathbb F}}
\newcommand{\G}{{\mathbb G}}

\renewcommand{\P}{{\mathbb P}}
\newcommand{\Q}{{\mathbb Q}}

\newcommand{\Z}{{\mathbb Z}}
\newcommand{\bfZ}{{\mathbf Z}}
\newcommand{\bfO}{{\mathbf O}}
\newcommand{\bfC}{{\mathbf C}}

\newcommand{\Irr}{{\mathrm{Irr}}}
\newcommand{\triv}{{\mathds{1}}}

\newcommand{\CC}{\mathbb C}

\newcommand{\Out}{\mathrm {Out}}

\newcommand{\PSp}{\mathrm{PSp}}
\newcommand{\GO}{\mathrm{O}}

\newcommand{\Gm}{{\mathbb G}_m}

\newcommand{\Aut}{\mathrm{Aut}}

\newcommand{\SU}{\mathrm{SU}}

\newcommand{\GL}{\mathrm{GL}}
\newcommand{\SL}{\mathrm{SL}}
\newcommand{\SO}{\mathrm{SO}}

\newcommand{\PSL}{\mathrm{PSL}}
\newcommand{\Co}{\mathsf{Co}}
\newcommand{\Suz}{\mathsf{Suz}}
\newcommand{\Alt}{\mathsf {A}}

\newcommand{\Swan}{{\mathsf {Swan}}}

\newcommand{\Wild}{{\mathsf {Wild}}}
\newcommand{\Tame}{{\mathsf {Tame}}}

\begin{document}
\title[
$2.{\sf{Co}}_1$ and  $6.\Suz$]{
A rigid local system with monodromy group the big Conway group 
$2.\Co_1$ and two others  with monodromy group the Suzuki group $6.{\sf{Suz}}$}
\author{Nicholas M. Katz, Antonio Rojas-Le\'{o}n, and Pham Huu Tiep}
\address{Department of Mathematics, Princeton University, Princeton, NJ 08544, USA}
\email{nmk@math.princeton.edu}
\address{Departamento de \'{A}lgebra, Universidad de Sevilla, c/Tarfia s/n, 41012 Sevilla, Spain}
\email{arojas@us.es}
\address{Department of Mathematics, Rutgers University, Piscataway, NJ 08854, USA}
\email{tiep@math.rutgers.edu}
\thanks{The second author was partially supported by MTM2016-75027-P (Ministerio de Econom\'ia y Competitividad) and FEDER. The third author gratefully acknowledges the support of the NSF (grant 
DMS-1840702).}

\maketitle

\begin{abstract} 
In the first three sections, we develop some basic facts about hypergeometric sheaves on the multiplicative group $\G_m$ in characteristic $p >0$. In the fourth and fifth sections, we specialize to quite special classses of hypergeomtric sheaves. We give relatively ``simple" formulas for their trace functions, and a criterion for them to have finite monodromy. In the next section, we prove that three of them have finite monodromy groups.We then give some results on finite complex linear groups.
 We next use these group theoretic results to show that one of our local systems, of rank $24$ in characteristic $p=2$, has the big Conway group $2.\Co_1$, in its irreducible orthogonal representation of degree $24$ as the automorphism group of the Leech lattice, as its arithmetic and geometric monodromy groups. Each of the other two, of rank $12$ in characteristic $p=3$, has the Suzuki group $6.\Suz$, in one of its irreducible representations of degree $12$ as the $\Q(\zeta_3)$-automorphisms of the Leech lattice, as its arithmetic and geometric monodromy groups. In the final section, we pull back these local systems by $x 
 \mapsto x^N$ maps to $\A^1$, and show that after pullback their arithmetic and geometric monodromy groups remain the same. Sadly the Leech lattice makes no appearance in 
our arguments.
\end{abstract}

\tableofcontents

\section*{Introduction}
In the first three sections, we develop some basic facts about hypergeometric sheaves on the multiplicative group $\G_m$ in characteristic $p >0$. In the fourth and fifth sections, we specialize to quite special classses of hypergeometric sheaves. We give relatively ``simple" formulas for their trace functions, and a criterion for them to have finite monodromy. In the next section, we prove that three of them have finite monodromy groups. We then give some results on finite complex linear groups.
 We next use these group theoretic results to show that one of our local systems, of rank $24$ in characteristic $p=2$, has the big Conway group $2.\Co_1$, in its irreducible orthogonal representation of degree $24$ as the automorphism group of the Leech lattice, as its arithmetic and geometric monodromy groups. Each of the other two, of rank $12$ in characteristic $p=3$, has the Suzuki group $6.\Suz$, in one of its irreducible representations of degree $12$ as the $\Q(\zeta_3)$-automorphisms of the Leech lattice, as its arithmetic and geometric monodromy groups. In the final section, we pull back these local systems by $x 
 \mapsto x^N$ maps to $\A^1$, and show that after pullback their arithmetic and geometric monodromy groups remain the same. Sadly the Leech lattice makes no appearance in 
our arguments.

\section{Primitivity}
We consider, in characteristic $p > 0$, a $\overline{\Q}_{\ell}$ ($\ell \neq p$) -hypergeometric sheaf $\sH$ of type $(n,m)$, with $n > m > 0$, thus
$$\sH=\sH yp(\psi,\chi_1,\ldots ,\chi_n;\rho_1,\ldots ,\rho_m).$$
Here $\psi$ is a nontrivial additive character of some finite extension $\F_q/\F_p$, and the $\chi_i$ and $\rho_j$ are (possibly
trivial) multiplicative characters of $\F_q^\times$, with the proviso that no $\chi_i$ is any $\rho_j$. One knows \cite[8.4.2, (1)]{Ka-ESDE} that such an $\sH$ is lisse on $\G_m$, geometrically irreducible, and on $\G_m/\overline{\F_q}$ has Euler characteristic $-1$. Its local monodromy at $0$ is of finite order if and only if the $\chi_i$ are pairwise distinct, in which case that local monodromy is their direct sum $\oplus_i \chi_i$, cf. \cite[8.4.2, (5)]{Ka-ESDE}.
 Its local monodromy
at $\infty$ is of finite order if and only the $\rho_j$ are pairwise distinct, in which case that local monodromy is the direct sum of
$\oplus_j \rho_j$ with a totally wild representation $\Wild_{n-m}$ of rank $n-m$ and Swan conductor one, i.e. it has all $\infty$-breaks $1/(n-m)$. The necessity results from the fact that any repetition of the $\rho_j$ produces nontrivial Jordan blocks. The sufficiency is given by the following lemma.

\begin{lem}\label{I finite}If the  $\rho_j$ are pairwise distinct, then the $I(\infty)$ action on $\sH$ factors through a finite quotient of  $I(\infty)$.
\end{lem}
\begin{proof}The $I(\infty)$ representation is the direct sum of the $n-m$ Kummer sheaves $\sL_{\rho}$ , together with a wild part $\Wild_{n-m}$ of rank $n-m$ and Swan conductor one, with all breaks $1/(n-m)$. This wild part ${\rm Wild}_{n-m}$ is $I(\infty)$-irreducible (because all its slopes are $1/(n-m)$). 
The action of $I(\infty)$ is thus completely reducible. Because this action is the restriction of an action of the
decomposition group $D(\infty)$, the local monodromy theorem assures us that on an open normal subgroup of 
$I(\infty)$, the representation becomes unipotent. But it remains completely reducible, so it becomes trivial.
\end{proof}

\begin{prop}\label{wheninduced}
Suppose that $\sH$ is  geometrically induced, i.e. that there exists a smooth connected curve 
$U/\overline{\F_q}$, a finite \'{e}tale map $\pi:U \rightarrow \G_m/\overline{\F_q}$ of degree $d \ge 2$, a lisse sheaf $\sG$ on $U$,
and an isomorphism $\sH \cong \pi_\star \sG$. Then up to isomorphism we are in one of the following situations.
\begin{itemize}
\item[(i)] {\rm (}{\bf Kummer induced}{\rm )} $U=\G_m$, $\pi$ is the $N^{\mathrm {th}}$ power map $x \mapsto x^N$ for some $N \ge 2$ prime to $p$ with $N|n$ and $N|m$,  $\sG$ is a hypergeometric sheaf of type $(n/N, m/N)$, and the lists of $\chi_i$ and of $\rho_j$ are each stable under multiplication by any chararacter $\Lambda$ of order dividing $N$.
\item[(ii)]{\rm (}{\bf Belyi induced}{\rm )} $U =\G_m \setminus \{1\}$, $\pi$ is either $x \mapsto x^A(1-x)^B$ or is $x \mapsto x^{-A}(1-x)^{-B}$,
$\sG$ is $\sL_{\Lambda(x)}\otimes \sL_{\sigma(x-1)}$ for some multiplicative characters $\Lambda$ and $\sigma$, and one of the following holds:

\begin{enumerate}[\rm(a)]
\item Both $A,B$ are prime to $p$, but $A+B =d_0p^r$ with $p \nmid d_0$ and $r \ge1$. In this case $\pi$ is  $x \mapsto x^A(1-x)^B$,
the $\chi_i$ are all the $A^{\mathrm {th}}$ roots of $\Lambda$ and all $B^{\mathrm {th}}$ roots of $\sigma$, 
and the $\rho_j$ are all the ${d_0}^{\mathrm {th}}$ roots of $(\Lambda \sigma)^{1/p^r}$.

\item $A$ is prime to $p$, $B=d_0p^r$ with $p \nmid d_0$ and $r \ge1$.  In this case $\pi$ is  $x \mapsto   x^{-A}(1-x)^{-B}$,
the $\chi_i$ are all the $(A+B)^{\mathrm {th}}$ roots of  $\Lambda \sigma$, and the $\rho_j$ are  all  the $A^{\mathrm {th}}$ roots of $\Lambda$ and 
all the ${d_0}^{\mathrm {th}}$ roots of $\sigma^{1/p^r}$.

\item $B$ is prime to $p$, $A=d_0p^r$ with $p \nmid d_0$ and $r \ge1$.  In this case $\pi$ is  $x \mapsto   x^{-A}(1-x)^{-B}$,
the $\chi_i$ are all the $(A+B)^{\mathrm {th}}$ roots of  $\Lambda \sigma$, and the $\rho_j$ are all the $B^{\mathrm {th}}$ roots of $\sigma$ together 
with all the ${d_0}^{\mathrm {th}}$ roots of $\Lambda^{1/p^r}$.
\end{enumerate}
\end{itemize}

\end{prop}
\begin{proof} If $\sH$ is $\pi_\star \sG$, then we have the equality of Euler Poincar\'{e} characteristics
$$EP(U,\sG) =EP(\G_m/\overline{\F_p},\pi_\star \sG) =EP(\G_m/\overline{\F_p},\sH)=-1.$$
Denote by $X$ the complete nonsingular model of $U$, and by $g_X$ its genus. Then $\pi$ extends to a finite flat map of $X$ to $\P^1$, and the Euler-Poincar\'{e} formula gives
$$-1=EP(U,\sG) = \rank(\sG)(2-2g_X -\#(\pi^{-1}(0) )- \#(\pi^{-1}(\infty) )) $$
$$- \sum_{w \in \pi^{-1}(0)}\Swan_w(\sG)- \sum_{w \in \pi^{-1} (\infty)}\Swan_w(\sG).$$
This shows that $g_X=0$, otherwise the coefficient of $\rank(\sG)$ is too negative. Then the sum
$$\#(\pi^{-1}(0) )+ \#(\pi^{-1}(\infty) )\le 3,$$
for the same reason. As each summand is strictly positive, we have one of two cases. Case (i) is 
$$\#(\pi^{-1}(0) )= \#(\pi^{-1}(\infty))=1.$$
On the source $X=\P^1$, we may assume $\pi^{-1}(0) =0$ and $\pi^{-1}(\infty) = \infty$.

Case (ii) is that, after possibly interchanging $0$ and $\infty$ on the target $\G_m$ by $x \mapsto 1/x$, we
have $$\#(\pi^{-1}(0) )=2, \ \#(\pi^{-1}(\infty))=1.$$
On the source $X=\P^1$, we may assume $\pi^{-1}(0) =\{0,1\}$ and $\pi^{-1}(\infty) = \infty$.
We then have that $\sG$ is lisse of rank one on $\P^1 \setminus \{0,1,\infty\}$, and everywhere tame, so it is $\sL_{\Lambda(x)}\otimes \sL_{\sigma(x-1)}$ for some multiplicative characters $\Lambda$ and $\sigma$,

We first treat case (i). Here $\pi$ is a finite \'{e}tale map from $\G_m/\overline{\F_q}$ to itself of degree $\ge 2$, which sends $0$ to itself and $\infty$ to itself, so necessarily (a nonzero constant
multiple of) the $N^{\mathrm {th}}$ power map, for some $N\ge 2$ prime to $p$. In this case the Euler-Poincar\'{e} formula shows that
$$\Swan_0(\sG)+\Swan_\infty(\sG)=1.$$
The lisse sheaf $ \sG$ is geometrically irreducible (because its direct image $\pi_\star \sG \cong \sH$ is). Therefore \cite[8.5.3.1]{Ka-ESDE} $\sG$ is itself
a hypergeometric sheaf, and $[N]_\star \sG \cong \sH$. In this case of Kummer induction, the rest follows from \cite[8.9.1 and 8.9.2]{Ka-ESDE}.

We now turn to case (ii). The map $\pi: \A^1 \setminus \{0,1\} \rightarrow \G_m$ is given by (a nonzero constant
multiple of) a polynomial $\pi(x)=f(x)=x^A(1-x)^B$ for some integers $A, B \ge 0$. This map being finite \'{e}tale insures that at least one of $A$ or $B$ is prime to $p$ (otherwise $f(x)$ is a $p^{\mathrm {th}}$ power). If either $A$ or $B$ vanishes, then possibly after $x \mapsto 1-x$ we have 
$B=0$, and we are in case (1), with $N=A$. Thus both $A$ and $B$ are strictly positive integers, at least one of which is prime to $p$.

The polynomial $f(x)=x^A(1-x)^B$ defines a finite \'{e}tale map from $\A^1 \setminus \{0,1\}$ to $\G_m$ if and only if the derivative $f'(x)$ has all its zeroes in the set $\{0,1\}$. Let us say that a zero outside  $\{0,1\}$ is a ``bad" zero. We readily calculate
$$f'(x) = \biggl(\frac{A}{x} + \frac{-B}{1-x}\biggr)f(x) =\biggl(\frac{A - (A+B)x}{x(1-x)}\biggr)f(x).$$
If $A+B$ is $0$ mod $p$, there are no bad zeroes. This will be subcase (a).  If $A+B$ nonzero mod $p$, then there is a zero at $x=A/(A+B)$. For this not to be a bad zero, either $A$ must be $0$ mod $p$, or $A$ must be $A+B$ mod $p$, i.e., either $A$ or $B$ must be $0$ mod $p$. These are subcases (b) and (c).

In subcase (a), we readliy compute the tame characters occurring in local monodromies at $0$ and at $\infty$ of $\pi_\star \sG$, with
$\sG =\sL_{\Lambda(x)}\otimes \sL_{\sigma(x-1)}$. In subcases (b) and (c), we do the same, now using $\pi(x):=\frac{1}{x^A(1-x)^B}$.
We know that $\pi_\star \sG$ has Euler Poincar\'{e} characteristic $-1$. 
If there are no tame characters that occur both at $0$ and at $\infty$, this data
determines \cite[8.5.6]{Ka-ESDE}, up to multiplicative translation, the geometrically irreducible hypergeometric sheaf which is the direct image. [If there are tame characters in common, this direct image is geometrically reducible \cite[8.4.7]{Ka-ESDE}. Being semisimple, it is the sum of Kummer sheaves $\sL_\chi$ for each $\chi$ in common, and a geometrically irreducible hypergeometric sheaf of lower rank.]
\end{proof}

\begin{cor}\label{pprim}
If an irreducible hypergeometric sheaf $\sH$ of type $(n,1)$ with $n \ge 2$ is geometrically induced, then its rank is a power of $p$.
\end{cor}
\begin{proof}It cannot be Kummer induced of degree $N \ge 2$, because $N$ must divide $\gcd(n,1)$. In case (2), subcases (b) and (c), there are at least two tame characters at $\infty$. In subcase (a), there is just one tame character at $\infty$ precisely when $A+B=n$ is a power of $p$ (i.e., when $d_0=1$ in that subcase).
\end{proof}
  
  \begin{cor}\label{ppprim}
An irreducible hypergeometric sheaf $\sH$ of type $(n,m)$ with $n > m >1$ and $n$ a power of $p$ is not geometrically induced.
\end{cor}
\begin{proof}It cannot be Kummer induced of degree $d \ge 2$, because $d$ is prime to $p$ but divides the rank $n$ of $\sH$. In case (ii), we must be in subcase (a), otherwise the rank is prime to $p$. In subcase (a), there is just one tame character at $\infty$, because $d_0=1$ in that subcase.
\end{proof}

 \begin{rmk}As Sawin has pointed out to us, for $\pi(x) =\frac{1}{x^{Nq}(1-x)}$, $\pi_\star \triv$ is the direct sum of $\triv$ with
 $$\sH yp(\psi, \mbox{all }\chi \mbox{ nontrivial with } \chi^{Nq+1}=\triv;  \mbox{all }\rho \mbox{ with } \rho^N=\triv).$$
 This last sheaf is thus "almost" induced from rank one, and hence has finite geometric monodromy.
 In particular, for $N=1$, 
 $$\sH yp(\psi, \mbox{ all }\chi \mbox{ nontrivial with } \chi^{q+1}=\triv;  \triv)$$
is ``almost" induced and has finite geometric monodromy.
\end{rmk}  
            
\section{Tensor indecomposability}
      
Over a field $k$, a representation $\Phi: G \to \GL(V)$ of a group $G$ is called {\it tensor decomposable} if there exists a $k$-linear isomorphism
$V \cong A\otimes_k B$ with both $A$, $B$ of dimension $\ge 2$, such that $\Phi(G) \leq \GL(A)\otimes_k \GL(B)$, the latter
being the image of $\GL(A)\times \GL(B)$ in $\GL(A\otimes_k B)$  by the map $(\phi,\rho)\mapsto  \phi \otimes \rho$.
  
In this situation, it is well known (also see the proof of Theorem \ref{indecompose} for more detail) that both $A$ and $B$ can be given the structure of projective representations of $G$, in such a way that the $k$-linear isomorphism $V \cong A\otimes_k B$ becomes an isomorphism of projective representations.
 
In reading the literature, it is important to distinguish this notion from the stronger notion of {\it linearly tensor decomposable}  that both $A$ and $B$ are 
$kG$-modules such that  the $k$-linear isomorphism $V \cong A\otimes_k B$ becomes an isomorphism of $kG$-modules.
  
 We use the term {\it tensor indecomposable} to mean ``not tensor decomposable" (and so ``tensor indecomposable'' is stronger than
 ``linearly tensor indecomposable'', cf. Remark \ref{2ex}).
  
The target of this section is the following theorem.

\begin{thm}\label{tensor} 
 In characteristic $p > 0$ and with $\ell \neq p$, a $\overline{\Q_\ell}$-hypergeometric sheaf
 $$\sH=\sH yp(\psi,\chi_1,\ldots ,\chi_n;\rho_1,\ldots ,\rho_m).$$
 of the type in Lemma \ref{I finite}, i.e. one of type $n > m >0$ with the ``downstairs" characters $\rho_i$ pairwise distinct, is tensor indecomposable as a representation of $\pi_1(\G_m/\overline{\F_p})$ if either of the following conditions holds:
 \begin{enumerate}[\rm(i)]
\item $n$ is neither $4$ nor an even power of $p$.
\item $n$ is an even power of $p$ and $m > 1$.
\end{enumerate} 
\end{thm}

\begin{rmk}\label{howto}
In order to show that a representation of a group $G$ is tensor indecomposable, it suffices to exhibit a subgroup $H$ of $G$ such that
the restriction to $H$ of the representation is tensor indecomposable as a representation of $H$. We will do this by taking the subgroup $I(\infty)$. In view of Lemma \ref{I finite}, $I(\infty)$ acts through a finite quotient group. In Theorem \ref{indecompose} below, we argue directly with this finite quotient group. In the Appendix, we give another approach to this same result. There we again use Lemma \ref{I finite}, this time combined with the fact that $I(\infty)$ has cohomological dimension $\le 1$ (in the suitable profinite world) to show first that if the representation
is tensor decomposable then in fact it is linearly tensor decomposable, and then we show that this is impossible under either
of the stated hypotheses.
\end{rmk}

We begin with the fraction field $K$ of a henselian discrete valuation ring $R$ whose residue field $k$ is  algebraically closed of characteristic $p > 0$, and consider a separable closure $K^{sep}$ of $K$. Then we have $I:=\Gal( K^{sep}/K)$ the inertia group, and $P \lhd I$ the $p$-Sylow subgroup of $I$. We fix a prime $\ell \neq p$, an algebraic closure $\overline{\Q}_{\ell}$ of $\Q_\ell$, and work in the category of continuous, finite dimensional $\overline{\Q}_{\ell}$-representations of $I$ (which we will call simply ``representations of $I$''). Note that any finite quotient group $J$ of $I$ is a finite group, with normal Sylow $p$-subgroup (which we will also denote by $P$) and with cyclic quotient $J/P$.

\begin{thm}\label{indecompose}
Let $J$ be a finite group, with normal Sylow $p$-subgroup $P$ and with cyclic quotient $J/P$. Let $V$ be a finite-dimensional 
$\C J$-module which is the direct sum $T \oplus W$ of a nonzero tame part $T$ (i.e., one on which $P$ acts trivially) and of an
irreducible submodule $W$ which is totally wild (i.e., one in which $P$ has no nonzero invariants). Suppose that one of the following conditions
holds.
\begin{enumerate}[\rm(a)]
\item $\dim(V)$ is neither $4$ nor an even power of $p$.
\item $\dim(V)$ is an even power of $p$ and $\dim(T) > 1$.
\end{enumerate} 
Then $J$ does not stabilize any decomposition $V = A \otimes B$ with $\dim(A), \dim(B) > 1$.
\end{thm}

\begin{proof}
(i) Suppose that $J$ fixes a decomposition $V = A \otimes B$ with $\dim(A), \dim(B) > 1$. Fix a basis $(e_1, \ldots,e_k)$ of 
$A$ and a basis $(f_1, \ldots,f_l)$ of $B$, so that $(e_i \otimes f_j \mid 1 \leq i \leq k,1 \leq j \leq l)$ is a basis of $V$.
Let $\Phi: J \to \GL_{\dim(V)}(\C)$ denote the matrix representation of $J$ on $V$ with respect to this basis. Then for each $g \in J$, we can
find matrices $\Theta(g) \in \GL_{\dim(A)}(\C)$ and $\Psi(g) \in \GL_{\dim(B)}(\C)$ such that 
\begin{equation}\label{eq-ind1}
  \Phi(g) = \Theta(g) \otimes \Psi(g).
\end{equation}
Note that if $X$ and $Y$ are invertible matrices (of possibly different sizes) over any field $\F$ so that $X \otimes Y$ is the identity matrix, then
$X$ and $Y$ are scalar matrices (of the corresponding sizes), inverses to each other. It follows that if $X \otimes Y = X' \otimes Y'$ for some invertible matrices
$X,X'$ of the same size and invertible matrices $Y,Y'$ of the same size, then $X' = \gamma X$ and $Y' = \gamma^{-1}Y$ for some 
$\gamma \in \F^{\times}$.  

Now, for any $g, h \in J$, by \eqref{eq-ind1} we have
$$\Theta(g)\Theta(h) \otimes \Psi(g)\Psi(h) = (\Theta(g) \otimes \Psi(g))(\Theta(h)\otimes \Psi(h)) = \Phi(g)\Phi(h) = \Phi(gh) = \Theta(gh) \otimes \Psi(gh).$$
By the above observation, $\Theta(gh)=\gamma(g,h)\Theta(g)\Theta(h)$ for some $\gamma(g,h) \in \C^{\times}$, i.e. the map 
$\Theta:g \mapsto \Theta(g)$ gives a projective representation of $J$, with factor set $\gamma$. We also have that 
$\Psi(gh) = \gamma(g,h)^{-1}\Psi(g)\Psi(h)$, and so $\Psi:g \mapsto \Psi(g)$ is a projective representation of $J$, with factor set $\gamma^{-1}$.
Hence, for a fixed universal cover $\hat J$ of $J$, we can lift  $\Theta$ and $\Psi$ to linear representations of $\hat J$. Thus we can 
view $A$ as a $\hat J$-module with character $\alpha$, and $B$ as a $\hat J$-module with character $\beta$. We can also inflate $V$ to a
$\hat J$-module, with character $\varphi$.

\smallskip
(ii) Recall that $J \cong \hat J/Z$ for some $Z \leq \bfZ(\hat J)$. Let $\hat P$ be the full inverse image of $P$ in $\hat J$, so that 
$\hat P/Z \cong P$, and let $Q := \bfO_p(\hat P)$. Note that $Q \lhd \hat J$ and $\hat P = Q \times \bfO_{p'}(Z)$. In particular, $\hat P$ 
acts trivially on $\Irr(Q)$. 

Recall the assumption that the $J$-module $W$ is irreducible. Let $\lambda_1$ be an irreducible constituent of the $P$-character
afforded by $W$, and let 
$J_1$ be the stabilizer of $\lambda_1$ in $J$. Since $J_1/P$ is cyclic, $\lambda_1$ extends to $J_1$ and any irreducible character 
of $J_1$ lying above $\lambda_1$ restricts to $\lambda_1$ over $P$, see e.g. \cite[(11.22), (6.17)]{Is}. It follows by Clifford theory
that the $P$-module $W$ affords the character $\lambda_1 + \ldots + \lambda_s$, where $\{\lambda_1, \ldots ,\lambda_s\}$ is a
$J/P$-orbit on $\Irr(P)$. We will now inflate these characters to $\hat P$-characters, also denoted $\lambda_1, \ldots,\lambda_s$, with
$Z$ and $\bfO_{p'}(Z)$ in their kernels, and then have that  
\begin{equation}\label{eq-ind2}
  \varphi|_Q = c \cdot 1_Q + \sum^s_{i=1}\lambda_i, \mbox{ with }\{\lambda_1, \ldots ,\lambda_s\} \mbox{ a }\hat J\mbox{-orbit on }\Irr(Q) \mbox{ and }
  c \in \Z_{\geq 1}. 
\end{equation}   
Now write 
\begin{equation}\label{eq-ind3}
  \alpha|_Q = a \cdot 1_Q + \sum^{m}_{i=1}\alpha_i,~\beta|_Q = b \cdot 1_Q + \sum^{m}_{i=1}\beta_j, 
\end{equation}
where $a,b,m,n \in \Z_{\geq 0}$, and $\alpha_i,\beta_j \in \Irr(Q) \smallsetminus \{1_Q\}$ not necessarily distinct. Since $\alpha$ is a 
$\hat J$-character and $Q \lhd \hat J$, $\{\alpha_1, \ldots,\alpha_m\}$ (if non-empty) is $\hat J$-stable, and similarly $\{\alpha_1, \ldots,\alpha_m\}$ is 
$\hat J$-stable if non-empty. In what follows we will refer to these two facts as $\hat J$-stability.

\smallskip
(iii) We will use the equality $\varphi|_Q = (\alpha|_Q)(\beta|_Q)$ to derive a contradiction. First we consider the case $a, m > 0$.

Suppose in addition that $b,n > 0$. Then $\varphi|_Q$ involves $b\sum^m_{i=1}\alpha_i + a\sum^n_{j=1}\beta_j$, and so, by $\hat J$-stability, it contains 
at least two $Q$-characters, each being a sum over some $\hat J$-orbit on $\Irr(Q) \smallsetminus \{1_Q\}$. This contradicts \eqref{eq-ind2}.

Now assume that $b > 0$ but $n = 0$. Then $\varphi|_Q = ab \cdot 1_P + b\sum^m_{i=1}\alpha_i$. Comparing the multiplicity of 
$\alpha_1$ in $\varphi|_Q$ and using \eqref{eq-ind2}, we see that $b=1$, and so $\dim(B) = \beta(1) = b = 1$, a contradiction.

Next we assume that $b=0$, so that $n > 0$. Then 
$$\varphi|_Q = a\sum^n_{j=1}\beta_j + \sum_{i,j}\alpha_i\beta_j.$$
Comparing the multiplicity of $\beta_1$ in $\varphi|_Q$ and using \eqref{eq-ind2}, we see that $a=1$ and moreover all 
$\beta_1, \ldots,\beta_n$ are pairwise distinct, whence $\{\beta_1, \ldots,\beta_n\}$ is a $\hat J$-orbit by $\hat J$-stability. This in turn implies 
by \eqref{eq-ind2} that $\{\beta_1, \ldots,\beta_n\} = \{\lambda_1, \ldots,\lambda_s\}$ and so 
\begin{equation}\label{eq-ind4}
  \alpha_i\beta_j = d_{ij}1_Q \mbox{ for some }d_{ij} \in \Z_{\geq 1} \mbox{ and for all }i,j.
\end{equation}  
Since $\alpha_i,\beta_j \in \Irr(Q)$, we observe that the multiplicity of $1_Q$ in $\alpha_i\beta_j$ is $0$ if $\beta_j \neq \overline{\alpha}_i$ and 
$1$ otherwise. Hence \eqref{eq-ind4} can happen only when  $\beta_j = \overline{\alpha}_i$ for all $i,j$ and moreover $\alpha_i(1)=1=\beta_j(1)$. If 
$n \geq 2$, we would then have $\beta_1 = \overline{\alpha}_1 = \beta_2$, a contradiction. So $n=1$ and $\dim(B) = n\beta_1(1) = 1$, again a 
contradiction.

\smallskip
(iv) In view of (iii), we have shown that $am=0$ and so $bn=0$ by symmetry. 

Assume in addition that $m=0$, so that $a > 0$. If $n=0$, then \eqref{eq-ind3} 
implies $\varphi|_Q = ab \cdot 1_Q$, contradicting \eqref{eq-ind2}. If $n > 0$, then $b=0$, and 
$\varphi|_Q = a\sum^n_{j=1}\beta_j$, again contradicting \eqref{eq-ind2}.

Thus we must have $a=0$, and so $b=0$ by symmetry. Now, according to $\hat J$-stability, $\alpha|_Q$ is, say $e$ times the sum 
over the $\hat J$-orbit $\{\alpha_1, \ldots ,\alpha_k\}$ of $\alpha_1 \in \Irr(Q)$.  As mentioned above, $\hat P = Q \times \bfO_{p'}(Z)$ acts 
trivially on $\Irr(Q)$, so only the cyclic group $\langle x \rangle \hat J/\hat P \cong J/P$ acts on $\Irr(Q)$. Note that, in any transitive action of 
any finite abelian group, all the point stabilizers are the same. Thus, if $\hat J_1$ is the unique subgroup of $\hat J$ of index $k$ that contains $\hat P$,
then $\hat J_1$ is the stabilizer of $\alpha_1$; moreover, we can write $\alpha_i = \alpha_1^{x^{i-1}}$ for $1 \leq i \leq k$, so that
\begin{equation}\label{eq-ind5}
  \alpha|_Q = e\sum^k_{i=1}\alpha_1^{x^{i-1}}.
\end{equation}  
The same argument applies to $\beta_Q$. Furthermore, since $1_Q$ contains in $(\alpha|_Q)(\beta|_Q)$, we may assume that 
$\beta_1 = \overline{\alpha}_1$. As $\alpha_1$ and $\overline{\alpha}_1$ have the same stabilizer in $\hat J$, we see that 
the $\hat J$-orbit of $\beta_1$ is exactly 
$$\{\overline{\alpha}_i = \overline{\alpha}_1^{x^{i-1}} \mid 1 \leq i \leq k\},$$
whence  
\begin{equation}\label{eq-ind6}
  \beta|_Q = f\sum^k_{i=1}\overline{\alpha}_1^{x^{i-1}}.
\end{equation}  
for some $f \in \Z_{\geq 1}$.

\smallskip
(v) Consider the case $k \geq 2$. Then $\varphi|_Q$ contains $ef\alpha_1\overline{\alpha}_1^x$ with $\alpha_1 \neq \alpha_1^x$. 
The latter implies that no irreducible constituent of $ef\alpha_1\overline{\alpha}_1^x$ can be $1_Q$, and so $ef=1$ by \eqref{eq-ind2}. 
Now, $\varphi|_Q$ contains the $\hat J$-stable character 
$$\Sigma:= \sum^{k-1}_{i=1}\alpha_1^{x^{i-1}}\overline{\alpha}_1^{x^i} + \alpha_1^{x^{k-1}}\overline{\alpha}_1,$$
and $[\varphi|_Q,1_Q]_Q = k$ by \eqref{eq-ind5} and \eqref{eq-ind6}. So \eqref{eq-ind2} implies that $\sum^s_{i=1}\lambda_i$ is 
contained in $\Sigma$ and $\varphi|_Q$ is contained in $k\cdot 1_Q+\Sigma$. Denoting $d:=\alpha_1(1)$ and comparing degrees, we then get
$$k^2d^2 =\alpha(1)\beta(1) \leq k+\Sigma(1) = k+kd^2.$$
As $k \geq 2$, we conclude that $k=2$, $d=1$, $\dim(V) = \chi(1) = 4$. In this case, $k=2$ divides the order of the 
$p'$-group $J/P$, so $p \neq 2$. This contradicts both (a) and (b). 

\smallskip
(vi) We have shown that $k=1$, so that $\alpha|_Q = e\alpha_1$ and $\beta|_Q = f\overline{\alpha}_1$.  Now if $\alpha_1(1) = 1$,
then $\varphi|_Q = ef \cdot 1_Q$, contradicting \eqref{eq-ind2}. Hence $\alpha_1(1) > 1$. In this case, we have that 
$\alpha_1\overline{\alpha}_1$ is a character of degree $> 1$ that contains $1_Q$ with multiplicity $1$, and so 
$\alpha_1\overline{\alpha}_1$ contains $1_Q+\mu$ for some $1_Q \neq \mu \in \Irr(Q)$. Comparing the multiplicity of $\mu$ 
using \eqref{eq-ind2}, we see that $ef=1$. Thus $\chi(1) = \alpha_1(1)^2$, and so it is an even power of $p$, since 
$\alpha_1$ is an irreducible character of the $p$-group $P$. Furthermore, from \eqref{eq-ind2} we see that
$1=c \geq \dim(T)$, and so $\dim(T) = 1$.
This possibility also contradicts both (a) and (b).
\end{proof}

\begin{rmk}\label{2ex}
As shown in \cite[3.2 and 3.6]{Ka-CC} (or can be seen on the example of the dihedral group of order $2p$), 
there are $V$ of dimension $4$ which are tensor decomposable when $p > 2$.
Furthermore, there also tensor decomposable examples in any dimension $p^{2m}$ with tame part of dimension $1$. Indeed, consider 
the group $Q \rtimes C$ with $Q$ extraspecial of order $p^{1+2m}$ and $C$ cyclic of order $p^{2m}-1$ that acts 
transitively on $\Irr(P) \smallsetminus \{1_P\}$, where $P := Q/\bfZ(Q)$. One can show 
that $Q \rtimes C$ has a complex module $W$ that affords a faithful irreducible character $\alpha$ of degree $p^m$, and 
$(\alpha\overline{\alpha})|_Q$ is trivial on $\bfZ(Q) \cong C_p$ and equal to the regular character of $P$. This implies 
that $W \otimes W^*$ has tame part of dimension $1$ and irreducible totally wild part of dimension $p^{2m}-1$. (Also note 
that $W \otimes W^*$ is indecomposable {\it as $P \rtimes C$-module}, even though $P \rtimes C$ preserves this tensor decomposition.)
\end{rmk}


\section{The image of $I(\infty)$}
In this section, we concentrate on the wild part $$W=W(\psi,\chi_1,\ldots  ,\chi_n;\rho_1,\ldots  ,\rho_m)$$ of the $I(\infty)$ representation attached to a hypergeometric sheaf
$$\sH=\sH yp(\psi,\chi_1,\ldots  ,\chi_n;\rho_1,\ldots  ,\rho_m)$$
of type $(n,m)$ with $n>m \ge 0$. We recall from \cite[8.1.14]{Ka-ESDE} that for given $\psi$ the isomorphism class of $W$ as $I(\infty)$ representation depends only on the tame character $\prod_i\chi_i/\prod_j\rho_j$ and its rank $N:=n-m$.

\begin{lem}\label{inertiaimage}Suppose that $N:=n-m$ is prime to $p$. If $N$ is odd, suppose that $\prod_i\chi_i/\prod_j\rho_j = \triv$. If $N$ is even, suppose that $\prod_i\chi_i/\prod_j\rho_j = \chi_2$.
Denote by $f$ the multiplicative order of $p$ in $(\Z/N\Z)^\times$, so that $\F_{p^f}$ is the extension $\F_p(\mu_{N})$ of $\F_p$ obtained by adjoining the $N$'th roots of unity. The image of the wild inertia group $P(\infty)$ is isomorphic to (the Pontryagin dual of the additive group of)  $\F_{p^f}$, acting as the direct sum $$\oplus_{\zeta \in \mu_N(\F_{p^f})}\sL_{\psi_N(\zeta x)}$$
of the $N$ characters
$ x \mapsto \psi_{\F_{p^f}}(N\zeta x)$. 
The quotient group $I(\infty)/P(\infty)$ acts through its quotient $\mu_N(\F_{p^f})$ by permuting these characters: $\alpha \in \mu_N(\F_{p^f})$ maps $\sL_{\psi_N(\zeta x)}$ to $\sL_{\psi_N(\alpha \zeta x)}$. In particular, a primitive $N$'th root of unity cyclically permutes these $N$ characters.
\end{lem}

\begin{proof}From  \cite[8.1.14]{Ka-ESDE}, we see that our $W$ occurs as the $I(\infty)$ representation attached to the Kloosterman sheaf
$$\sK l(\psi;\mbox{all\ characters\ of\ order\ dividing }N), $$
which in turn is known \cite[5.6.2]{Ka-GKM} to be geometrically isomorphic to the Kummer direct image $[N]_\star(\sL_{\psi_N(x)})$. This direct image is $I(\infty)$-irreducible because the $N$ multiplicative translates of $\sL_{\psi_N(x)}$ by $\mu_N(\F_{p^f})$ are pairwise $I(\infty)$-inequivalent.
The determination of the image of $P(\infty)$ is done exactly as in \cite[Lemma 1.2]{Ka-RL-T-Co3}. That the quotient group $I(\infty)/P(\infty)$ acts through its quotient $\mu_N(\F_{p^f})$ in the asserted way is implicit in the very definition of inducing a character from a normal subgroup of cyclic index $N$.
\end{proof}

\section{A particular class of hypergeometric sheaves}
We remain in characterstic $p > 0$, with a chosen $\ell \neq p$, and a chosen nontrivial additive character $\psi$ of $\F_p$. Fix two  integers $A,B \ge 3$ with $\gcd(A,B)=1$ and both $A,B$ prime to $p$. We denote by
$$\sH yp(\psi,A\times B;\triv)$$
 the hypergeometric sheaf whose ``upstairs" characters are the $(A-1)(B-1)$ characters of the form
$\chi \rho$ with $\chi \neq \triv, \chi^A=\triv$ and $\rho \neq \triv, \rho^B=\triv$, and whose ``downstairs" character is the single character $\triv$. It is defined on $\G_m/\F_q$ for any finite extension of $\F_p$ containing the $AB$'th roots of unity. One knows \cite[8.8.13]{Ka-ESDE} that $\sH yp(\psi,A\times B;\triv)$ is pure of weight $(A-1)(B-1)$, and geometrically irreducible. 

\begin{lem}\label{det} The determinant of $\sH yp(\psi,A\times B;\triv)$ is geometrically trivial.
\end{lem}
\begin{proof}Because both $A,B \ge 3$, the rank $(A-1)(B-1)$ is $\ge 4$. Hence the wild part $\rm{Wild}$ of the $I(\infty)$ representation has dimension $(A-1)(B-1) - 1 \ge 3 >2$, so all slopes $<1$, and hence $\det(\rm{Wild})$ must be tame. Therefore $\det(\sH yp)$ is tame,
and must be equal to the product of its $(A-1)(B-1) $ ``upstairs" characters, the $\chi_i \rho_j$. At least one of $A,B$ must be odd (because they are relatively prime), and therefore the product of the $\chi_i \rho_j$ is trivial.
\end{proof}

\begin{lem}\label{selfdual}$\sH yp(\psi,A\times B;\triv)$ is geometrically self dual precisely in the case $p=2$, and in that case it is orthogonally self dual.
\end{lem}
\begin{proof}This is immediate from \cite[8.8.1 and 8.8.2]{Ka-ESDE}, because, as noted above, at least one of $A,B$ is odd, and
hence $\sH yp(\psi,A\times B;\triv)$ has even rank, but only one tame character ``downstairs", namely $\triv$. And it is obvious that the ``upstairs" characters, the $\chi_i \rho_j$, are stable by complex conjugation (indeed by all of $\Gal(\overline{\Q}/\Q)$).
\end{proof}

In terms of Kubert's $V$ function, we have the following criterion for finite monodromy.
\begin{lem}\label{finitemonocrit}Let $\F_q$ be a finite extension of $\F_p$ containing the $AB^{\mathrm {th}}$ roots of unity. Then the Tate twist
$$\sH yp(\psi,A\times B;\triv)((A-1)(B-1)/2)$$
has finite geometric and arithmetic monodromy groups if and only if, for all $x \in (\Q/\Z)_{\mathrm {prime \ to \ }p}$, we have
$$V(ABx) + V(x) + V(-x) \ge V(Ax) +V(Bx).$$
Equivalently, since this trivially holds for $x=0$, the criterion is that for all nonzero $x \in (\Q/\Z)_{\mathrm{prime \ to \ }p}$, we have
$$V(ABx) + 1 \ge V(Ax) +V(Bx).$$
\end{lem}
\begin{proof}Entirely similar to the proof of \cite[Lemma 2.1]{Ka-RL-T-Co2}, using the Hasse-Davenport relation to simplify
the Mellin transform calculation \cite[8.2.8]{Ka-ESDE} of the trace function of $\sH yp(\psi,A\times B;\triv)((A-1)(B-1)/2)$. 
\end{proof}

Although it is possible to descend $\sH yp(\psi,A\times B;\triv)$ to $\G_m/\F_p$, using \cite[8.8]{Ka-GKM}, we will instead give a ``more computable" descent to $\G_m/\F_p(\zeta_A)$. 

\begin{lem}\label{descent}Denote by $\chi_i$ the $A-1$ nontrivial characters of order dividing $A$. The lisse sheaf $\sH(\psi,A\times B)$ on $\G_m/\F_p(\zeta_A)$ whose trace function at a point $s \in K^\times$, $K/\F_p(\zeta_A)$ a finite extension, is given by
$$s \mapsto \biggl(\frac{-1}{\#K}\biggr)^{A -1}\sum_{(t_i )_i \in \G_m(K)^{A -1}}\psi\bigl(\frac{-\prod_i t_i}{s}\bigr)\prod_i \chi_i(t_i)
   \sum_{(x_i)_i \in \A^1(K)^{A -1}}\psi_K\biggl(B(\sum_i x_i) -\sum_ix_i^B/t_i\biggr)$$
is a descent to $\G_m/\F_p(\zeta_A)$ of a constant field twist of $\sH yp(\psi,A\times B;\triv)$.
\end{lem}
\begin{proof}Separate the numerator characters into packets $\chi_i \times (\rm{all\ allowed\ }\rho)$, indexed by the $A-1$ nontrivial $\chi_i$. Each 
of these packets is the list of characters for $\sL_{\chi_i}\otimes \sK l(\psi, \rho \neq \triv, \rho^B=\triv)$. The multiplicative $_{\star, !}$ 
convolution of  $\sL_{\psi(-1/x)}$ with all of these is, by definition, the hypergeometric sheaf $\sH yp(\psi,A\times B;\triv)$. 

As proven in \cite[Lemma 1.2]{Ka-RL-T-Co2}, the Kloosterman sheaf $$\sK l(\psi,\mbox{all\ nontrivial\ characters\ of\  order\  dividing\ }B)$$
has a descent to (a constant field twist of) the local system $\sB _0$ on $\G_m/\F_p$ whose trace function is
$$t \in K^\times \mapsto - \sum_{x \in K}\psi_K(-x^B/t +Bx).$$ Convolving these $\sL_{\chi_i}\otimes \sB_0$ gives the assertion. 
\end{proof}

\begin{lem}\label{weight}The lisse sheaf $\sH(\psi,A\times B)$ is pure of weight zero.
\end{lem}
\begin{proof}The sheaf $\sH yp(\psi,A\times B;\triv)$ is pure of weight $(A-1)(B-1)$. In replacing each $\sK l(\psi,\mbox{all\ nontrivial\ characters\ of\  order\  dividing\ }B)$ by $\sB_0$, we save weight $(B-3)$ in each replacement, so all in all we save weight $(A-1)(B-3)$. The division by $(\#K)^{A-1}$ brings the weight down to zero.
\end{proof}

\begin{lem}\label{tracefield}The trace function of $\sH(\psi,A\times B)$ takes values in the field $\Q(\zeta_p)$.
\end{lem}
\begin{proof}
In the formula for the trace, we write the final summation 
$$\sum_{(x_i)_i \in \A^1(K)^{A -1}}\psi_K\biggl(B(\sum_i x_i) -\sum_ix_i^B/t_i\biggr)$$
as 
$$\prod_{1 \le i \le A-1} \biggl(\sum_{x \in K}\psi_K(Bx-x^B/t_i)\biggr),$$
a symmetric function of the $t_i$. The factor $\psi(- (\prod_i t_i)/s)$ is also a symmetric function of the $t_i$.
So the formula for the
trace at $s \in K^\times$ has the shape
$$\sum_{(t_i )_i \in \G_m(K)^{A -1}}(\prod_i \chi_i(t_i))(\mbox{Symmetric function of }(t_1,\ldots  ,t_{A-1})).$$
If we precompose with an automorphism of $\G_m^{A -1}$ given by a permutation of the variables, this sum (indeed any
sum over $\G_m(K)^{A -1}$) does not change. But the effect of this on our sum is to correspondingly permute the $\chi_i$.
Thus in the formula for the trace, the sum does not change under any permutation of the $\chi_i$. When we apply an element
of $\Gal(\overline{\Q}/\Q(\zeta_p))$ to the sum, its only effect is to permute the $\chi_i$ (it permutes them among themselves because  they are all the nontrivial characters
of order dividing $A$, so as a set are Galois stable, even under $\Gal(\overline{\Q}/\Q)$), or equivalently to permute the variables. Thus our sum is invariant under $\Gal(\overline{\Q}/\Q(\zeta_p))$, so lies in $\Q(\zeta_p))$.
\end{proof}

\begin{lem}\label{det}If $p=2$, then $\sH(\psi,A\times B)$ has $\Q$-valued trace function, and we have inclusions
$$G_{geom} \subset \SO_{(A-1)(B-1)}(\overline{\Q}_{\ell}),\  \ G_{geom} \lhd G_{arith} \subset \GO_{(A-1)(B-1)}(\overline{\Q}_{\ell}).$$
If we pass to the quadratic extension of $\F_p(\zeta_A)$, then we have
$$G_{geom} \lhd G_{arith} \subset \SO_{(A-1)(B-1)}(\overline{\Q}_{\ell}).$$
\end{lem}
\begin{proof}From Lemma \ref{selfdual}, we know that  $\sH(\psi,A\times B)$ is, geometrically, orthogonally self dual. From Lemma \ref{tracefield} and Lemma \ref{weight}, we know that $\sH(\psi,A\times B)$ is pure of weight zero and has $\Q$-valued traces. This implies that $\sH(\psi,A\times B)$ is arithmetically self dual. Because $\sH(\psi,A\times B)$ is geometrically (and hence arithmetically) irreducible, 
the autoduality of $\sH(\psi,A\times B)$ is unique up to a nonzero scalar factor. Being orthogonal geometrically, the autoduality of $\sH(\psi,A\times B)$ must be orthogonal. Thus both $G_{geom}$ and $G_{arith}$ lie in $\GO_{(A-1)(B-1)}(\overline{\Q}_{\ell})$. By Lemma \ref{det}, we then get that $G_{geom}$ lies in $\SO_{(A-1)(B-1)}(\overline{\Q}_{\ell})$. 

If $G_{arith}$ lies in $\SO$, we are done. If not, then the determinant of $\sH(\psi,A\times B)$ is geometrically trivial, but takes values in $\pm 1$, so must be the constant field twist $(-1)^{\deg}$, which disappears when we pass to the quadratic extension of  $\F_p(\zeta_A)$.
\end{proof}

\begin{lem}\label{detagain}If $p\neq 2$, then $\sH(\psi,A\times B)$ is not geometrically self dual, and  we have inclusions
$$G_{geom} \subset \SL_{(A-1)(B-1)}(\overline{\Q}_{\ell}),\  \ G_{geom} \lhd G_{arith} \subset \GL_{(A-1)(B-1)}(\overline{\Q}_{\ell}).$$
If we pass to the degree $2p$ extension of $\F_p(\zeta_A)$, then we have
$$G_{geom} \lhd G_{arith} \subset \SL_{(A-1)(B-1)}(\overline{\Q}_{\ell}).$$
\end{lem}
\begin{proof}From Lemma \ref{det}, we know that  $\sH(\psi,A\times B)$  has geometrically trivial determinant, which gives the first assertion. From Lemma \ref{tracefield} and Lemma \ref{weight}, we know that $\sH(\psi,A\times B)$ is pure of weight zero and has $\Q(\zeta_p)$-valued traces. Therefore its determinant is of the form $\alpha^{\deg}$, for some $\alpha \in \Q(\zeta_p)$ which is a unit outside of the unique place over $p$, and all of whose complex absolute values are $1$. Thus $\alpha$ is a root of unity in $ \Q(\zeta_p)$, so of order dividing $2p$. So after passing to the degree $2p$ extension of $\F_p(\zeta_4)$, the determinant becomes arithmetically trivial as well.
\end{proof}

\section{A second class of hypergeometric sheaves}
We remain in characterstic $p > 0$, with a chosen $\ell \neq p$, and a chosen nontrivial additive character $\psi$ of $\F_p$. Fix an  integer $A \ge 7$ which is prime to $p$. We denote by $\phi(A)$ the Euler $\phi$ function: 
$$\phi(A) := \#(\Z/A\Z)^\times ={\rm \ number\ of\ characters\ of \ order  \ }A.$$
We denote by
$$\sH yp(\psi,A^\times ;\triv)$$
 the hypergeometric sheaf whose ``upstairs" characters are the $\phi(A)$ characters of order $A$, and whose ``downstairs" character is the single character $\triv$. It is defined on $\G_m/\F_q$ for any finite extension of $\F_p$ containing the $A^{\mathrm {th}}$ roots of unity. One knows \cite[8.8.13]{Ka-ESDE} that $\sH yp(\psi,A^\times;\triv)$ is pure of weight $\phi(A) $, and geometrically irreducible. 

\begin{lem}\label{detbis} The determinant of $\sH yp(\psi,A^\times;\triv)$ is geometrically trivial.
\end{lem}
\begin{proof}Because  $A \ge 7$, the rank $\phi(A)$ is $\ge 4$. Hence the wild part $\Wild$ of the $I(\infty)$ representation has dimension $ \ge 3 >1$, so all slopes $<1$, and hence $\det(\Wild)$ must be tame. Therefore $\det(\sH yp)$ is tame,
and must be equal to the product of its $\phi(A)$ ``upstairs" characters. Their product must be trivial, because they are stable by inversion and  (because $A > 2$) none of them is $\chi_2$.
\end{proof}

We now explain the criterion for finite monodromy in terms of Kubert's $V$ function. For simplicity, we will state it only in the case when $A$ is divisible by precisely two distinct primes $p_1$ and $p_2$. Denote  by $\Phi_N(X) \in \Z[X]$ the cyclotomic polynomial for the primitive $N^{\mathrm {th}}$ roots of unity. Then
$$\Phi_A(X) =\frac{(X^A-1)(X^{A/(p_1p_2)}-1)}{(X^{A/p_1}-1)(X^{A/p_2}-1)}.$$

\begin{lem}\label{finitemonocritbis}Let $\F_q$ be a finite extension of $\F_p$ containing the $A^{\mathrm {th}}$ roots of unity. Suppose that $A$ is divisible by precisely two distinct primes $p_1$ and $p_2$.
 Then the Tate twist
$$\sH yp(\psi,A^\times;\triv)(\phi(A)/2)$$
has finite geometric and arithmetic monodromy groups if and only if, for all $x \in (\Q/\Z)_{\mathrm{prime\  to\ }p}$, we have
$$V(Ax) + V(Ax/(p_1p_2)) + V(-x) \ge V(Ax/p_1) +V(Ax/p_2).$$
\end{lem}
\begin{proof}Entirely similar to the proof of \cite[Lemma 2.1]{Ka-RL-T-Co2}, using the Hasse-Davenport relation to simplify
the Mellin transform calculation \cite[8.2.8]{Ka-ESDE} of the trace function of $\sH yp(\psi,A^\times;\triv)(\phi(A)/2)$. 
\end{proof}

We now specialize further.

\begin{lem}\label{descentbis}Suppose that the characteristic $p$ is odd, and that $A=4B$ with $B$ an odd prime, $B \neq p$.
The lisse sheaf $\sH(\psi,(4B)^\times)$ on $\Gm/\F_p(\zeta_4)$ whose trace function at a point $s \in K^\times$, $K/\F_p(\zeta_4)$ a finite extension, is given by
$$s \mapsto \biggl(\frac{-1}{\#K}\biggr)^2\sum_{(u,v) \in \G_m(K)^2}\psi\biggl(\frac{-uv}{s}\biggr)\chi_4(u)\overline{\chi_4}(v)
   \sum_{(x,y) \in \A^1(K)^{2}}\psi_K\biggl(B(x+y) -\frac{x^B}{u} -\frac{y^B}{v}\biggr)$$
is a descent to $\G_m/\F_p(\zeta_4)$ of a constant field twist of $\sH yp(\psi,A^\times;\triv)$.

\end{lem}
\begin{proof}Entirely similar to the proof of Lemma \ref{descent}.
\end{proof}

\begin{lem}\label{weightbis}The lisse sheaf $\sH(\psi,(4B)^\times)$ is pure of weight zero.
\end{lem}
\begin{proof}Entirely similar to the proof of Lemma \ref{weight}.
\end{proof}

\begin{lem}\label{tracefieldbis}The trace function of  $\sH(\psi,(4B)^\times)$ takes values in the field $\Q(\zeta_p)$.
\end{lem}
\begin{proof}Entirely similar to the proof of Lemma \ref{tracefield}.
\end{proof}

\begin{lem}\label{detter}For $\sH(\psi,(4B)^\times)$ on $\Gm/\F_p(\zeta_4)$, we have
$$G_{geom} \subset \SL_{2(B-1)}(\overline{\Q}_{\ell}).$$
After passing to the degree $2p$ extension of $\F_p(\zeta_4)$, we have
$$G_{geom} \lhd G_{arith} \subset \SL_{2(B-1)}(\overline{\Q}_{\ell}).$$
\end{lem}
\begin{proof}The first assertion is just Lemma \ref{detbis}, that $\det(\sH(\psi,(4B)^\times))$ is geometrically constant. In view of Lemmas \ref{tracefieldbis} and \ref{weightbis}, $\det(\sH(\psi,(4B)^\times))$ is of the form $\alpha^{\deg}$ for some $\alpha \in \Q(\zeta_p)$ which is a unit outside of the unique place over $p$, and all of whose complex absolute values are $1$. Thus $\alpha$ is a root of unity in $ \Q(\zeta_p)$, so of order dividing $2p$. So after passing to the degree $2p$ extension of $\F_p(\zeta_4)$, the determinant becomes arithmetically trivial as well.
\end{proof}

\section{Theorems of finite monodromy}
\begin{thm}\label{3 x 13} In characteristic $p=2$, the lisse sheaf $\sH(\psi,3 \times 13)$ on $\G_m/\F_4$ has finite $G_{arith}$ and finite $G_{geom}$.
\end{thm}
\begin{proof}By Lemma \ref{finitemonocrit}, we must show that
$$V(39x) + 1 \ge V(3x) + V(13x),$$
for all nonzero $x \in (\Q/\Z)_{\mathrm {prime\ to \ }p}$.
 If $x\in\frac{1}{39}{\mathbb Z}$ we check it by hand. If $39x \neq 0$, then by the change of variable $x \mapsto -x$ and the relation
 $V(x)+V(-x)=1$ for $x \neq 0$, it is equivalent to
$$
V(39x)\leq V(3x)+V(13x)
$$
which, applying the formula $V(3x)+1=V(x)+V(x+\frac{1}{3})+V(x+\frac{2}{3})$, is equivalent to
$$
V\left(13x+\frac{1}{3}\right)+V\left(13x+\frac{2}{3}\right)\leq V(x)+V\left(x+\frac{1}{3}\right)+V\left(x+\frac{2}{3}\right).
$$

In terms of the $[-]_{r}:=[-]_{2,r}$ function \cite[\S 4]{Ka-RL}, we need to show that, for all {\bf even} $r\geq 2$ and all integers $0<x<2^r-1$ we have
\begin{equation}\label{V-main1}
\left[13x+\frac{2^r-1}{3}\right]_r+\left[13x+\frac{2(2^r-1)}{3}\right]_r\leq [x]_r+\left[x+\frac{2^r-1}{3}\right]_r+\left[x+\frac{2(2^r-1)}{3}\right]_r.
\end{equation}

For even $r\geq 2$, let $A_r=\frac{2^r-1}{3}$ and $B_r=\frac{2(2^r-1)}{3}$. For odd $r\geq 1$, let $A_r=\frac{2^{r+1}-1}{3}$ and $B_r=\frac{2^r-2}{3}$. Note that, for $1\leq s<r$, we have
$$
\begin{array}{ccc}
 A_r=2^sA_{r-s}+A_s, & B_r=2^sB_{r-s}+B_s & \mbox{if $s$ is even}  \\
 A_r=2^sB_{r-s}+A_s, & B_r=2^sA_{r-s}+B_s & \mbox{if $s$ is odd} 
\end{array}
$$
For a non-negative integer $x$, let $[x]$ denote the sum of the 2-adic digits of $x$.

\begin{lem}
 Let $r\geq 1$ and let $0\leq x <2^r$ an integer. Then
 $$
[13x+A_r]+[13x+B_r]\leq [x]+[x+A_r]+[x+B_r]+4.
$$
Moreover, if $r\geq 4$ and the first four digits of $x$ are not $0100$, $1000$ or $1001$, then
$$
[13x+A_r]+[13x+B_r]\leq [x]+[x+A_r]+[x+B_r]+2.
$$
If $x<2^{r-2}$ (i.e. the first two of the $r$ $2$-adic digits of $x$ are $0$) then
$$
[13x+A_r]+[13x+B_r]\leq [x]+[x+A_r]+[x+B_r]+1.
$$
Finally, if the first four digits of $x$ are $1010$, then
$$
[13x+A_r]+[13x+B_r]\leq [x]+[x+A_r]+[x+B_r].
$$
\end{lem}

\begin{proof}
 We proceed by induction on $r$: for $r\leq 14$ one checks it by computer. Let $r\geq 15$ and $0\leq x< 2^r$, and consider the $2$-adic expansion of $x$. By adding leading 0's as needed, we will assume that it has exactly $r$ digits.
 
 In all cases below we will follow one of these two procedures: in the first one, for some $1\leq s\leq r-4$, we write $x=2^sy+z$ with $y<2^{r-s}$, $z<2^s$. Assume $s$ is even (otherwise, just interchange $A_{r-s}$ and $B_{r-s}$ below). Let $C$ be the total number of digit carries in the sums $13x+A_r=2^s(13y+A_{r-s})+(13z+A_s)$ and $13x+B_r=2^s(13y+B_{r-s})+(13z+B_s)$ and $D$ the total number of digit carries in the sums $x+A_r=2^s(y+A_{r-s})+(z+A_s)$ and $x+B_r=2^s(y+B_{r-s})+(z+B_s)$, and let $\lambda_s(z):=[13z+A_s]+[13z+B_s]-[z]-[z+A_s]-[z+B_s]$. If $C-D-\lambda_s(z)\geq 0$, then
 \begin{equation}\label{induct1}
[13x+A_r]+[13x+B_r] =[2^s(13y+A_{r-s})+(13z+A_s)]+[2^s(13y+B_{r-s})+(13z+B_s)]
\end{equation}
$$
=[13y+A_{r-s}]+[13z+A_s]+[13y+B_{r-s}]+[13z+B_s]-C\leq
$$
$$
\leq [y]+[y+A_{r-s}]+[y+B_{r-s}]+4+[z]+[z+A_s]+[z+B_s]+\lambda_s(z)-C\leq
$$
$$
\leq  [y]+[y+A_{r-s}]+[y+B_{r-s}]+4+[z]+[z+A_s]+[z+B_s]-D=
$$
$$
=[2^sy+z]+[2^s(y+A_{r-s})+(z+A_s)]+[2^s(y+B_{r-s})+(z+B_s)]+4=[x]+[x+A_r]+[x+B_r]+4
$$
by induction. Moreover, the first four digits of $x$ and $y$ are the same, so the better inequalities hold for $x$ whenever they do for $y$.

In the second procedure, for some $1\leq s\leq r-4$, we write $x=2^sy+z$ with $y<2^{r-s}$, $z<2^s$. Again we assume $s$ is even (otherwise, just interchange $A_{r-s}$ and $B_{r-s}$ below). For some $0<s'<s$ (which we also assume even without loss of generality) we find some $z'<2^{s'}$ such that the following conditions hold: if $z+A_s$ (respectively $z+B_s$, $13z+A_s$, $13z+B_s$) has $s+\alpha$ digits (resp. $s+\beta$, $s+\gamma$, $s+\delta$), then $z'+A_{s'}$ (resp. $z'+B_{s'}$, $13z'+A_{s'}$, $13z'+B_{s'}$) has $s'+\alpha$ digits (resp. $s'+\beta$, $s'+\gamma$, $s'+\delta$) and the first $\alpha$ digits of $z+A_s$ and $z'+A_{s'}$ (resp. the fist $\beta$ digits of $z+B_s$ and $z'+B_{s'}$, the first $\gamma$ digits of $13z+A_s$ and $13z'+A_{s'}$, the first $\delta$ digits of $13z+B_s$ and $13z'+B_{s'}$) coincide. Moreover, we require that $\lambda_s(z)\leq\lambda_{s'}(z')$.

Let $r'=r-s+s'$ and $x'=2^{s'}y+z'<2^{r'}$. Then the total number $C$ of digit carries in the sums $13x+A_r=2^s(13y+A_{r-s})+(13z+A_s)$ and $13x+B_r=2^s(13y+B_{r-s})+(13z+B_s)$ is the same as the total number of digit carries in the sums $13x'+A_{r'}=2^{s'}(13y+A_{r-s})+(13z'+A_{s'})$ and $13x'+B_{r'}=2^{s'}(13y+B_{r-s})+(13z'+B_{s'})$, and the total number $D$ of digit carries in the sums $x+A_r=2^s(y+A_{r-s})+(z+A_s)$ and $x+B_r=2^s(y+B_{r-s})+(z+B_s)$ is the same as the total number of digit carries in the sums $x'+A_{r'}=2^{s'}(y+A_{r-s})+(z'+A_{s'})$ and $x'+B_{r'}=2^{s'}(y+B_{r-s})+(z'+B_{s'})$, so we have
 \begin{equation}\label{induct2}
[13x+A_r]+[13x+B_r]=[2^s(13y+A_{r-s})+(13z+A_s)]+[2^s(13y+B_{r-s})+(13z+B_s)]
\end{equation}
$$
=[13y+A_{r-s}]+[13z+A_s]+[13y+B_{r-s}]+[13z+B_s]-C\leq
$$
$$
\leq [13y+A_{r-s}]+[13z'+A_{s'}]+[13y+B_{r-s}]+[13z'+B_{s'}]-C+
$$
$$
+[z]+[z+A_s]+[z+B_s]-[z']-[z'+A_{s'}]-[z'+B_{s'}]=
$$
$$
=[13x'+A_{r'}]+[13x'+B_{r'}]+[z]+[z+A_s]+[z+B_s]-[z']-[z'+A_{s'}]-[z'+B_{s'}]\leq
$$
$$
\leq [x']+[x'+A_{r'}]+[x'+B_{r'}]+4+[z]+[z+A_s]+[z+B_s]-[z']-[z'+A_{s'}]-[z'+B_{s'}]=
$$
$$
=[y]+[y+A_{r-s}]+[y+B_{r-s}]+4+[z]+[z+A_s]+[z+B_s]-D=
$$
$$
=[x]+[x+A_r]+[x+B_r]+4
$$
by induction. Moreover, the first four digits of $x$ and $x'$ are the same, so the better inequalities hold for $x$ whenever they do for $x'$.
 
 \bigskip
 \begin{enumerate}
 \item[{\it Case}  1:] $x\equiv 0 \mod 2$. We apply (\ref{induct1}) with $s=1$, so $z=0$ and $C=D=\lambda_s(z)=0$.
 
\item[{\it Case} 2:] The last three digits of $x$ are 001.

\item[{\it Case} 2a:] The last 4 digits of $x$ are 0001. Take $s=4$ in (\ref{induct1}), so $z=1=0001_2$. Then $D$ is clearly $0$ and $\lambda_s(z)=0$, so $C-D-\lambda_s(z)=C\geq 0$.

\item[{\it Case} 2b:] The last $5$ digits of $x$ are 01001. Take $s=3$, so $z=1=001_2$ and $y\equiv 1 \mod 4$. Here $A_3+1=6$ and $B_3+1=3$ are both $<8$, so $D=0$. On the other hand, $13+A_3=18=10010_2$ and the last two digits of $13y+B_{r-3}$ are $11$, so $C\geq 1$. Therefore $C-D-\lambda_s(z)\geq 1-1=0$.

\item[{\it Case} 2c:] The last six digits of $x$ are 011001. We can apply (\ref{induct2}) with $s=6$, $z=25=011001_2$, $s'=5$ and $z'=13=01101_2$, so $\lambda_s(z)=\lambda_{s'}(z')=2$.

\item[{\it Case} 2d:] The last seven digits of $x$ are 0111001. We can apply (\ref{induct2}) with $s=7$, $z=57=0111001_2$, $s'=6$ and $z'=31=011111_2$, so $\lambda_s(z)=\lambda_{s'}(z')=0$.

\item[{\it Case} 2e:] The last nine digits of $x$ are 001111001. We apply (\ref{induct2}) with $s=9$, $z=121=001111001_2$, $s'=2$ and $z'=1=01_2$, so $\lambda_s(z)=1<\lambda_{s'}(z')=3$.

\item[{\it Case} 2f:] The last ten digits of $x$ are 0101111001. We apply (\ref{induct1}) with $s=5$, so $z=25=11001_2$, $\lambda_s(z)=1$ and $D=1$. Since $13\cdot 25+A_5=101011010_2$ and the last five digits of $13y+B_{r-5}$ are $11001$, we get at least two digit carries in the sum $13x+A_r=2^5(13y+B_{r-5})+(13z+A_5)$, so $C-D-\lambda_s(z)\geq 2-1-1=0$.

\item[{\it Case} 2g:] The last ten digits of $x$ are 1101111001. We can apply (\ref{induct2}) with $s=10$, $z=889=1101111001_2$, $s'=7$ and $z'=109=1101101_2$, so $\lambda_s(z)=\lambda_{s'}(z')=2$.

\item[{\it Case} 2h:] The last nine digits of $x$ are 011111001. We can apply (\ref{induct2}) with $s=9$, $z=249=011111001_2$, $s'=8$ and $z'=121=01111001_2$, so $\lambda_s(z)=\lambda_{s'}(z')=1$.

\item[{\it Case} 2i:] The last nine digits of $x$ are 111111001. We apply (\ref{induct1}) with $s=5$, so $z=25=11001_2$, $\lambda_s(z)=1$ and $D=1$. Since $13\cdot 25+A_5=101011010_2$, $13\cdot 25+B_5=101001111_2$ and the last four digits of $13y+B_{r-5}$ and $13y+A_{r-5}$ are $1101$ and $1000$ respectively, we get at least one digit carry in each of the sums $13x+A_r=2^5(13y+B_{r-5})+(13z+A_5)$ and $13x+B_r=2^5(13y+A_{r-5})+(13z+B_5)$, so $C-D-\lambda_s(z)\geq 2-1-1=0$.

\item[{\it Case} 3:] The last $r-4$ digits of $x$ contain two consecutive 0's. Suppose that the right-most ones are located in positions $t-1,t$ for $t\geq 2$ (counting from the right). If $t=2$, $x$ is even and we apply case 1. If $t=3$ we apply case 2. Suppose that $t\geq 4$.

\item[{\it Case} 3a:] Either previous two digits are not $11$ or the next two digits are $11$. Take $s=t$ in (\ref{induct1}). Then $\lambda_s(z)\leq 1$, since $z<2^{s-2}$. Assume without loss of generality that $s$ is even. Both $z+A_s$ and $z+B_s$ are $<2^s$, so $D=0$. If $t=4$ and the two digits after the 0's are $10$ we can apply case 1. Otherwise, since we picked the right-most consecutive 0's, the following digits are at least $101$ (that is, $z\geq 2^{s-3}+2^{s-5}$). It follows that $13z+A_s$ and $13z+B_s$ both have $s+2$ digits, the first two being $10$ or $11$. Then, if the second-to-last digit of either $13y+A_{r-s}$ or $13y+B_{r-s}$ is $1$ (which is the case if the last two digits of $y$ are not $11$) or the last two digits of $y$ are $11$ and the first two digits of $13z+B_s$ are 11 (which is the case if the third and fourth digits of $z$ are 11), we get at least one digit carry in one of the sums $13x+A_r=2^s(13y+A_{r-s})+(13z+A_s)$ or $13x+B_r=2^s(13y+B_{r-s})+(13z+B_s)$, so $C\geq 1$ and $C-D-\lambda_s(z)\geq 0$.

\item[{\it Case} 3b:] The last five digits of $x$ are $00101$. We apply (\ref{induct1}) with $s=5$ and $z=5=00101_2$, so $D=0$ and $\lambda_s(z)=-1$. 

\item[{\it Case} 3c:] The last six digits of $x$ are $001011$. We apply (\ref{induct1}) with $s=6$ and $z=5=001011_2$, so $D=0$ and $\lambda_s(z)=0$.

\item[{\it Case} 3d:] The previous two digits are $11$, and the next four digits are $1010$. Take $s=t-2$ in (\ref{induct1}), so the last four digits of $y$ are 1100 and $\lambda_s(z)\leq 0$ by induction. If $z$ has no two consecutive $1$'s (that is, $z=101010...$), there are no digit carries in the sums $x+A_r=2^s(y+A_{r-s})+(z+A_s)$ and $x+B_r=2^s(y+B_{r-s})+(z+B_s)$, so $D=0$ and we are done. Otherwise $D=1$. Note that $13\cdot\overbrace{10...10}^{t\times 10}11+\overbrace{01...01}^{(t+1)\times 01}=1001\overbrace{00...0}^{2t-1}100$ so (if $s$ is even) $13z+A_s$ has $s+4$ digits, the first four being $1001$. Then there is one digit carry in the sum $13x+A_r=2^s(13y+A_{r-s})+(13z+A_s)$, so $C-D-\lambda_s(z)\geq 1-1-0=0$.

\item[{\it Case} 3e:] In all remaining cases, the previous two digits are $11$ and the next four are $1011$ (since we picked the right-most consecutive 0's). Take $s=t-4$ in (\ref{induct1}), then the last six digits of $y$ are $110010$ and the first two digits of $z$ are $11$. There is no digit carry in the sum $x+B_r=2^{s}(y+B_{r-s})+(z+B_{s})$ and at most three digit carries in the sum $x+A_r=2^{s}(y+A_{r-s})+(z+A_{s})$, so $D\leq 3$.

Suppose first that $s<6$. Then $\lambda_s(z)\leq 1$ and $\lambda_s(z)=1$ only for $s=4$, $z=13=1101_2$ and $s=5$, $z=25=11001_2$, $z=26=11010_2$ or $z=27=11011_2$, as one can check directly. On the other hand, $13z+A_{s}$ and $13z+B_{s}$ have $s$ digits, the first four being $1010$, $1011$, $1100$ or $1101$, and the last six digits of $13y+A_{r-s}$ and $13y+B_{r-s}$ are $011111$ and $110100$ respectively. Then we get at least three digit carries in the sum $13x+A_r=2^{s}(13y+A_{r-s})+(13z+A_{s})$. Moreover, in all cases where $\lambda_s(z)=1$ the first four digits of $13z+A_{s}$ are 1010 or 1011, so we get at least four digit carries in the sum above. In any case, $C-D-\lambda_s(z)\geq 3-3=0$.

If $s\geq 6$, then the first six digits of $z$ are at least $110101$, and the first four digits of $13z+A_{s}$ and $13z+B_{s}$ are $1011$, $1100$ or $1101$. If the first four digits of $13z+A_{s}$ are $1011$ we get five digit carries in the sum in the sum $13x+A_r=2^{s}(13y+A_{r-s})+(13z+A_{s})$. If the first four digits of $13z+A_{s}$ are $1100$ or $1101$, then so are the first four digits of $13z+B_{s}$, and we get at least three digit carries in the sum $13x+A_r=2^{s}(13y+A_{r-s})+(13z+A_{s})$ and at least two more in the sum $13x+B_r=2^{s}(13y+B_{r-s})+(13z+B_{s})$. In either case, we have $C-D-\lambda_s(z)\geq 5-3-2=0$. 

This concludes the proof of case 3.

\item[{\it Case} 4:] The last two digits of $x$ are $11$. Suppose that the last string of consecutive $1$'s has length $m\geq 2$.

\item[{\it Case} 4a:] $m\geq 5$, that is, the last five digits of $x$ are $11111$. We apply (\ref{induct1}) with $s=3$, so $z=7$. Then there is one digit carry in the sum $x+A_r=8(y+B_{r-3})+(7+A_3)$ and no digit carries in the sum $x+B_r=8(y+A_{r-3})+(7+B_3)$. If there is a $0$ before the five $1$'s, then we can assume that the last four digits of $y$ are $1011$ (otherwise we would have two consecutive 0's and we could apply case 2). Then $13\cdot 7+A_3=1100000_2$ and the last four digits of $13y+B_{r-3}$ are 1001, so there is one digit carry in the sum $13x+A_r=8(13y+B_{r-3})+(13\cdot 7+A_3)$. If there is a $1$ before the five $1$'s, then the last three digits of $13y+B_{r-3}$ are 101, so again there is (at least) one digit carry in the above sum. Therefore $C-D-\lambda_s(z)\geq 1-1-0=0$.

\item[{\it Case} 4b:] $m=4$, that is, the last five digits of $x$ are $01111$. We can assume that the previous digit is a $1$ (otherwise we apply case 2). We apply (\ref{induct1}) with $s=2$, so $z=3$ and $y$ ends with $1011$. Then there is one digit carry in the sum $x+B_r=4(y+B_{r-2})+(3+B_2)$ and no digit carries in the sum $x+A_r=4(y+A_{r-2})+(3+A_2)$. Also, $13\cdot 3+B_2=41=101001_2$, and the last four digits of $13y+B_{r-2}$ are $1001$, so there is one digit carry in the sum $13x+B_r=4(13y+B_{r-2})+(13\cdot 3+B_2)$. Therefore $C-D-\lambda_s(z)=1-1-0=0$.

\item[{\it Case} 4c:] $m=3$, that is, the last four digits of $x$ are $0111$. Again, we can assume that the previous digit is a $1$. We apply (\ref{induct1}) with $s=1$, so $z=1$ and $y$ ends with $1011$. Then there is one digit carry in the sum $x+A_r=2(y+B_{r-1})+(1+A_1)$ and no digit carries in the sum $x+B_r=2(y+A_{r-1})+(1+B_1)$. Also, $13\cdot 1+A_1=14=1110_2$, and the last four digits of $13y+B_{r-1}$ are $1001$, so there are at least four digit carries in the sum $13x+A_r=2(13y+B_{r-1})+(13\cdot 1+A_1)$. Therefore $C-D-\lambda_s(z)\geq 4-1-3=0$.

The remaining subcases of case 4 have $m=2$, that is, the last four digits of $x$ are $1011$.

\item[{\it Case} 4d:] The last six digits of $x$ are $111011$. We apply (\ref{induct1}) with $s=4$, so $z=11=1011_2$ and $y$ ends with $11$. Then there is one digit carry in the sum $x+B_r=16(y+B_{r-4})+(11+B_4)$ and no digit carries in the sum $x+A_r=16(y+A_{r-4})+(11+B_4)$. Also, $13\cdot 11+B_4=10011001_2$, and the last two digits of $13y+B_{r-4}$ are $01$, so there is at least one digit carry in the sum $13x+B_r=16(13y+B_{r-4})+(13\cdot 11+B_4)$. Therefore $C-D-\lambda_s(z)\geq 1-1-0=0$.

\item[{\it Case} 4e:] The last six digits of $x$ are $011011$. We apply (\ref{induct1}) with $s=3$, so $z=3=011_2$ and $y$ ends with $011$. Then there is one digit carry in the sum $x+A_r=8(y+B_{r-3})+(3+A_3)$ and no digit carries in the sum $x+B_r=8(y+A_{r-3})+(3+B_3)$. Also, $13\cdot 3+A_3=44=101100_2$, $13\cdot 3+B_3=41=101001_2$, and the last three digits of $13y+A_{r-3}$ and $13y+B_{r-3}$ are $100$ and $001$ respectively, so there is at least one digit carry in each of the sums $13x+A_r=8(13y+B_{r-3})+(13\cdot 3+A_3)$ and $13x+B_r=8(13y+A_{r-3})+(13\cdot 3+B_3)$. Therefore $C-D-\lambda_s(z)\geq 2-1-1=0$.

\item[{\it Case} 4f:] The last six digits of $x$ are $101011$. We can apply (\ref{induct2}) with $s=6$, $z=43=101011_2$, $s'=4$ and $z'=11=1011_2$, since $\lambda_s(z)=-1<0=\lambda_{s'}(z')$.

%

This ends the proof for case 4. It remains to check the case where $x$ ends with 01. By case 2, we can assume that the previous digit is a 1.

\item[{\it Case} 5:] The last four digits of $x$ are 0101. We apply (\ref{induct1}) with $s=3$, so $z=5=101_2$ and $y$ is even. Then there are no digit carries in either sum  $x+A_r=8(y+B_{r-3})+(5+A_3)$ or $x+B_r=8(y+A_{r-3})+(5+B_3)$, so $C-D-\lambda_s(z)=C+1> 0$.

\item[{\it Case} 6:] The last five digits of $x$ are 11101. We apply (\ref{induct1}) with $s=1$, so $z=1$ and $y$ ends with $1110$. Then there are no digit carries in the sums $x+A_r=2(y+B_{r-1})+(1+A_1)$ and $x+B_r=2(y+A_{r-1})+(1+B_1)$. Also, $13\cdot 1+B_1=13=1101_2$, and the last four digits of $13y+A_{r-1}$ are $1011$, so there are at least three digit carries in the sum $13x+B_r=2(13y+A_{r-1})+(13\cdot 1+B_1)$. Therefore $C-D-\lambda_s(z)\geq 3-0-3=0$.

In the remaining cases, the last six digits of $x$ are 101101.

\item[{\it Case} 7:] The last nine digits of $x$ are 111101101. We apply (\ref{induct1}) with $s=4$, so $z=13=1101_2$ and $y$ ends with $11110$. Then there are two digit carries in the sum $x+A_r=16(y+A_{r-4})+(13+A_4)$ and no digit carries in the sum $x+B_r=16(y+B_{r-4})+(13+B_4)$. Also, $13\cdot 13+A_4=10101110_2$, and the last five digits of $13y+A_{r-4}$ are $11011$, so there are at least three digit carries in the sum $13x+A_r=16(13y+A_{r-4})+(13\cdot 13+A_4)$. Therefore $C-D-\lambda_s(z)\geq 3-2-1=0$.

\item[{\it Case} 8:] The last nine digits of $x$ are 011101101. By case 2, we can assume that the previous digit is a 1. We apply (\ref{induct1}) with $s=6$, so $z=45=101101_2$ and $y$ ends with $1011$. Then there is one digit carry in the sum $x+B_r=64(y+B_{r-6})+(45+B_6)$ and no digit carries in the sum $x+A_r=64(y+A_{r-6})+(45+A_6)$. Also, $13\cdot 45+B_6=1001110011_2$ and the least four digits of $13y+B_{r-6}$ are 1001, so there are at least two digit carries in the sum $13x+B_r=64(13y+B_{r-6})+(13\cdot 45+B_6)$. Therefore $C-D-\lambda_s(z)\geq 2-1-1=0$.

\item[{\it Case} 9:] The last eight digits of $x$ are 01101101. By case 2, we can assume that the previous digit is a 1. We apply (\ref{induct2}) with $s=9$, $z=365=101101101_2$, $s'=6$ and $z'=45=101101_2$. Here $\lambda_s(z)=\lambda_{s'}(z')=1$.

\item[{\it Case} 10:] The last seven digits of $x$ are 0101101.By case 2, we can assume that the previous digit is a 1. We apply (\ref{induct2}) with $s=8$, $z=173=10101101_2$, $s'=4$ and $z'=11=1011_2$. Here $\lambda_s(z)=\lambda_{s'}(z')=0$.
\end{enumerate}
\end{proof}

\begin{cor}
Let $r\geq 2$ be even and let $0< x <2^r-1$ an integer. Then
 $$
[13x+A_r]_r+[13x+B_r]_r\leq [x]_r+[x+A_r]_r+[x+B_r]_r+5.
$$
\end{cor}

\begin{proof}
If $x=A_r=\frac{2^r-1}{3}$ or $x=B_r=\frac{2(2^r-1)}{3}$ it is obviuos. Otherwise, the $2$-adic expansion of $x$ contains two consecutive 0's or two consecutive 1's. 

In the first case, multiplying by a suitable power of 2 we can assume that the first two digits of $x$ are 00, that is, $x<2^{r-2}$. Then $x+A_r$ and $x+B_r$ are both $<2^r$, so
$$
[13x+A_r]_r+[13x+B_r]_r\leq[13x+A_r]+[13x+B_r]\leq
$$
$$
\leq[x]+[x+A_r]+[x+B_r]+4=[x]_r+[x+A_r]_r+[x+B_r]_r+4.
$$

In the second case, multiplying by a suitable power of 2 we can assume that the last two digits of $x$ are 11. Then $x+A_r$ ends with 10 and has at most $r+1$ digits. If $x+A_r<2^r$ then $[x+A_r]_r=[x+A_r]$, and if $x+A_r\geq 2^r$ then $[x+A_r]_r=[x+A_r-2^r+1]_r=[x+A_r-2^r+1]=[x+A_r]-1+1=[x+A_r]$. On the other hand, $x+B_r$ ends with 01 and has at most $r+1$ digits. If $x+B_r<2^r$ then $[x+B_r]_r=[x+B_r]$, and if $x+B_r\geq 2^r$ then $[x+B_r]_r=[x+B_r-2^r+1]_r=[x+B_r-2^r+1]=[x+B_r]-1$, since there is one digit carry in the sum $(x+B_r)+1$. In any case,
$$
[13x+A_r]_r+[13x+B_r]_r\leq[13x+A_r]+[13x+B_r]\leq
$$
$$
\leq[x]+[x+A_r]+[x+B_r]+4\leq[x]_r+[x+A_r]_r+[x+B_r]_r+5.
$$
\end{proof}

We can now finish the proof of Theorem \ref{3 x 13}. By the numerical Hasse-Davenport relation \cite[Lemma 2.10]{Ka-RL-T-Co3}, using that $A_{kr}=\frac{2^{kr}-1}{2^r-1}A_r$ and $B_{kr}=\frac{2^{kr}-1}{2^r-1}B_r$ we get, applying the previous corollary to $x':=\frac{2^{kr}-1}{2^r-1}x$,
$$
[13x'+A_{kr}]_{kr}+[13x'+B_{kr}]_{kr}\leq[x']_{kr}+[x'+A_{kr}]_{kr}+[x'+B_{kr}]_{kr}+5\Rightarrow
$$
$$
\Rightarrow
[13x+A_r]_r+[13x+B_r]_r\leq[x]_r+[x+A_r]_r+[x+B_r]_r+\frac{5}{k}
$$
and we conclude by taking $k\to\infty$.
\end{proof}

\begin{thm}\label{4 x 5} In characteristic $p=3$, the lisse sheaf $\sH(\psi,4 \times 5)$ on $\G_m/\F_9$ has finite $G_{arith}$ and finite $G_{geom}$.
\end{thm}
\begin{proof}By Lemma \ref{finitemonocrit}, we must show that
$$V(20x) + 1 \ge V(4x) + V(5x),$$
for all nonzero $x \in (\Q/\Z)_{\mathrm {prime\ to \ }p}$.
If $x\in\frac{1}{20}{\mathbb Z}$ we check it by hand, otherwise, just as in the proof of Theorem \ref{3 x 13}, it is equivalent to
$$
V(20x)\leq V(4x)+V(5x)
$$
which, applying the duplication formula, is equivalent to
$$
V\left(5x+\frac{1}{2}\right)+V\left(10x+\frac{1}{2}\right)\leq V(x)+V\left(x+\frac{1}{2}\right)+V\left(2x+\frac{1}{2}\right).
$$

In terms of the $[-]_{r}:=[-]_{3,r}$ function \cite[\S 4]{Ka-RL}, we need to show that, for all $r\geq 2$ and all integers $0<x<3^r-1$ we have
\begin{equation}\label{V-main2}
\left[5x+\frac{3^r-1}{2}\right]_r+\left[10x+\frac{3^r-1}{2}\right]_r\leq [x]_r+\left[x+\frac{3^r-1}{2}\right]_r+\left[2x+\frac{3^r-1}{2}\right]_r.
\end{equation}

For a non-negative integer $x$, let $[x]$ denote the sum of the 3-adic digits of $x$.

\begin{lem}
 Let $r\geq 1$ and let $0\leq x <3^r$ an integer. Then
 $$
\left[5x+\frac{3^r-1}{2}\right]+\left[10x+\frac{3^r-1}{2}\right]\leq [x]+\left[x+\frac{3^r-1}{2}\right]+\left[2x+\frac{3^r-1}{2}\right]+2.
$$
Moreover, if the first two digits of $x$ are not 10, 11 or 21, then we have the better inequality
$$
\left[5x+\frac{3^r-1}{2}\right]+\left[10x+\frac{3^r-1}{2}\right]\leq [x]+\left[x+\frac{3^r-1}{2}\right]+\left[2x+\frac{3^r-1}{2}\right].
$$
\end{lem}

\begin{proof}
 We proceed by induction on $r$: for $r\leq 7$ one checks it by computer. Let $r\geq 8$ and $0\leq x< 3^r$, and consider the $3$-adic expansion of $x$. By adding leading 0's as needed, we will assume that it has exactly $r$ digits. Let $A_r=\frac{3^r-1}{2}$.
 
 In all cases below we will follow one of these two procedures: in the first one, for some $s\leq r-2$, we write $x=3^sy+z$ with $y<3^{r-s}$, $3<2^s$. Let $C$ be the total number of digit carries in the sums $5x+A_r=3^s(5y+A_{r-s})+(5z+A_s)$ and $10x+A_r=3^s(10y+A_{r-s})+(10z+A_s)$ and $D$ the total number of digit carries in the sums $x+A_r=3^s(y+A_{r-s})+(z+A_s)$ and $2x+A_r=3^s(2y+A_{r-s})+(2z+A_s)$, and let $\lambda_s(z)=[5z+A_s]+[10z+A_s]-[z]-[z+A_s]-[2z+A_s]$. If $2(C-D)-\lambda_s(z)\geq 0$, then
 \begin{equation}\label{induct3}
[5x+A_r]+[10x+A_r]=[3^s(5y+A_{r-s})+(5z+A_s)]+[3^s(10y+A_{r-s})+(10z+A_s)]=
\end{equation}
$$
=[5y+A_{r-s}]+[5z+A_s]+[10y+A_{r-s}]+[10z+A_s]-2C\leq
$$
$$
\leq [y]+[y+A_{r-s}]+[2y+A_{r-s}]+2+[z]+[z+A_s]+[2z+A_s]+\lambda_s(z)-2C\leq
$$
$$
\leq  [y]+[y+A_{r-s}]+[2y+A_{r-s}]+2+[z]+[z+A_s]+[2z+A_s]-2D=
$$
$$
=[3^sy+z]+[3^s(y+A_{r-s})+(z+A_s)]+[3^s(2y+A_{r-s})+(2z+A_s)]+2=[x]+[x+A_r]+[2x+A_r]+2
$$
by induction. Moreover, the first two digits of $x$ and $y$ are the same, so the better inequality holds for $x$ whenever they do for $y$.

In the second procedure, for some $1\leq s\leq r-2$, we write $x=3^sy+z$ with $y<3^{r-s}$, $z<3^s$. For some $0<s'<s$ we find some $z'<3^{s'}$ such that the following conditions hold: if $z+A_s$ (respectively $2z+A_s$, $5z+A_s$, $10z+A_s$) has $s+\alpha$ digits (resp. $s+\beta$, $s+\gamma$, $s+\delta$), then $z'+A_{s'}$ (resp. $2z'+A_{s'}$, $5z'+A_{s'}$, $10z'+A_{s'}$) has $s'+\alpha$ digits (resp. $s'+\beta$, $s'+\gamma$, $s'+\delta$) and the first $\alpha$ digits of $z+A_s$ and $z'+A_{s'}$ (resp. the fist $\beta$ digits of $2z+A_s$ and $2z'+A_{s'}$, the first $\gamma$ digits of $5z+A_s$ and $5z'+A_{s'}$, the first $\delta$ digits of $10z+A_s$ and $10z'+A_{s'}$) coincide. Moreover, we require that $\lambda_s(z)\leq\lambda_{s'}(z')$.

Let $r'=r-s+s'$ and $x'=3^{s'}y+z'<3^{r'}$. Then the total number $C$ of digit carries in the sums $5x+A_r=3^s(5y+A_{r-s})+(5z+A_s)$ and $10x+A_r=3^s(10y+A_{r-s})+(10z+A_s)$ is the same as the total number of digit carries in the sums $5x'+A_{r'}=3^{s'}(5y+A_{r-s})+(5z'+A_{s'})$ and $10x'+A_{r'}=3^{s'}(10y+A_{r-s})+(10z'+A_{s'})$, and the total number $D$ of digit carries in the sums $x+A_r=3^s(y+A_{r-s})+(z+A_s)$ and $2x+A_r=3^s(2y+A_{r-s})+(2z+A_s)$ is the same as the total number of digit carries in the sums $x'+A_{r'}=3^{s'}(y+A_{r-s})+(z'+A_{s'})$ and $2x'+A_{r'}=3^{s'}(2y+A_{r-s})+(2z'+A_{s'})$, so we have
 \begin{equation}\label{induct4}
[5x+A_r]+[10x+A_r]=[3^s(5y+A_{r-s})+(5z+A_s)]+[3^s(10y+A_{r-s})+(10z+B_s)]
\end{equation}
$$
=[5y+A_{r-s}]+[5z+A_s]+[10y+A_{r-s}]+[10z+A_s]-2C\leq
$$
$$
\leq [5y+A_{r-s}]+[5z'+A_{s'}]+[10y+A_{r-s}]+[10z'+A_{s'}]-2C+
$$
$$
+[z]+[z+A_s]+[2z+A_s]-[z']-[z'+A_{s'}]-[2z'+A_{s'}]=
$$
$$
=[5x'+A_{r'}]+[10x'+A_{r'}]+[z]+[z+A_s]+[2z+A_s]-[z']-[z'+A_{s'}]-[2z'+A_{s'}]\leq
$$
$$
\leq [x']+[x'+A_{r'}]+[2x'+A_{r'}]+2+[z]+[z+A_s]+[2z+A_s]-[z']-[z'+A_{s'}]-[2z'+A_{s'}]=
$$
$$
=[y]+[y+A_{r-s}]+[2y+A_{r-s}]+2+[z]+[z+A_s]+[2z+A_s]-D=
$$
$$
=[x]+[x+A_r]+[2x+A_r]+2
$$
by induction. Moreover, the first two digits of $x$ and $x'$ are the same, so the better inequality holds for $x$ whenever they do for $x'$.
 
 \bigskip
 \begin{enumerate}
 \item[{\it Case} 1:] $x\equiv 0 \mod 3$. We apply (\ref{induct3}) with $s=1$, so $z=0$ and $C=D=\lambda_s(z)=0$.
 
\item[{\it Case} 2:] The last two digits of $x$ are $01$ or $02$. We apply (\ref{induct3}) with $s=2$, so $z\leq 2$ and $D=\lambda_s(z)=0$ (since $2z+A_2\leq 8<3^2$). Therefore $2(C-D)-\lambda_s(z)=2C\geq 0$.

\item[{\it Case} 3:] The last three digits of $x$ are $011$. We apply (\ref{induct3}) with $s=3$, so $z=4=011_3$ and $D=\lambda_s(z)=0$ (since $2z+A_3=21<3^3$). Therefore $2(C-D)-\lambda_s(z)=2C\geq 0$.

\item[{\it Case} 4:] The last three digits of $x$ are $111$. We apply (\ref{induct3}) with $s=1$, so $z=1=1_3$, $\lambda_s(z)=1$ and $D=0$ (since $2y+A_{r-1}$ ends with a 0). On the other hand, $10y+A_{r-1}$ ends with $22$ and $10z+A_1=102_3$, so $C\geq 1$. Therefore $2(C-D)-\lambda_s(z)\geq 2-1=1$.

\item[{\it Case} 5:] The last $r-1$ digits of $x$ contain the strings $00$ or $01$. Pick $s$ such that the first two digits of $z$ are $00$ or $01$. Then $D=0$ and $\lambda_s(z)\leq 0$, so $2(C-D)-\lambda_s(z)\geq 2C\geq 0$. If $s=r-1$ and the first digit of $x$ is not 1 then we have the better inequality (since $\lambda_1(0)=\lambda_1(2)=0$).

\item[{\it Case} 6:] The last $r-1$ digits of $x$ contain the string $10$. Pick $s$ such that the last digit of $y$ is $1$ and the first digit of $z$ is $0$. Then the last digit of $2y+A_{r-s}$ is 0, so $D=0$ and $\lambda_s(z)\leq 0$. Therefore $2(C-D)-\lambda_s(z)\geq 2C\geq 0$.

\item[{\it Case} 7:] $x$ contains one of the strings $1202$ or $2202$. Pick $s$ such that the last two digits of $y$ are $12$ or $22$ and the first two digits of $z$ are $02$. Then the last two digits of $2y+A_{r-s}$ are $02$ or $12$, so there is at most one digit carry in the sum $2x+A_r=3^s(2y+A_{r-s})+(2z+A_s)$, and $D\leq 1$. On the other hand, the last digit of $5y+A_{r-s}$ is 2, and $3^s<5z+A_s<3^{s+1}$, so there is at least one digit carry in the sum $5x+A_r=3^s(5y+A_{r-s})+(5z+A_s)$. Therefore $2(C-D)-\lambda\geq 2(1-1)-0=0$.

\item[{\it Case} 8:] The last four digits of $x$ are 0211. Using cases 5,6 and 7, we can assume that the previous three digits are 202. We apply (\ref{induct4}) with $s=7$, $z=1642=2020211_3$, $s'=5$ and $z'=184=20211_3$. Here $\lambda_{s}(z)=-2<0=\lambda_{s'}(z')$.

\item[{\it Case} 9:] The last four digits of $x$ are 1211. Let $t$ be the number of consecutive 1's before the 2. We apply (\ref{induct3}) with $s=3$, so $z=22=211_3$ and $\lambda_s(z)=2$. There are $t$ digit carries in the sum $x+A_r=27(y+A_{r-3})+(z+A_3)$ and no digit carry in $2x+A_r=27(2y+A_{r-3})+(2z+A_3)$, so $D=t$. On the other hand, the last $t$ digits of $5y+A_{r-3}$ and $10y+A_{r-3}$ are $\overbrace{22\ldots 22}^{t-1}0$ and $\overbrace{11\ldots 11}^{t-3}022$ respectively, and $5\cdot 22+A_3=123=11120_3$ and $10\cdot 22+A_3=233=22122_3$, so we get $t-1$ digit carries in the sum $5x+A_r=27(5y+A_{r-3})+(5z+A_3)$ and at least two more in the sum $10x+A_r=27(10y+A_{r-3})+(10z+A_3)$. Therefore $2(C-D)-\lambda_s(z)\geq 2(t+1-t)-2=0$.

\item[{\it Case} 10:] The last four digits of $x$ are 2211. We apply (\ref{induct3}) with $s=2$, so $z=4=11_3$, $\lambda_s(z)=2$ and the last two digits of $2y+A_{r-2}$ are 02, so $D=1$. On the other hand, $5y+A_{r-2}$ ends with $22$ and $5z+A_2=24=220_3$, so $C\geq 2$. Therefore $2(C-D)-\lambda_s(z)\geq 2-2=0$.

\item[{\it Case} 11:] The last three digits of $x$ are 021 or 022. Using cases 5,6 and 7, we can assume that the previous three digits are 202. We apply (\ref{induct4}) with $s=6$, $z=547=202021_3$ or $z=548=202022_3$, $s'=4$, and $z'=61=2021_3$ or $z'=62=2022_3$ respectively. Here $\lambda_s(z)=-3<-1=\lambda_{s'}(z')$ in the $z=547$ case and $\lambda_s(z)=-2<0=\lambda_{s'}(z')$ in the $z=548$ case.

\item[{\it Case} 12:] The last three digits of $x$ are $121$. This is similar to case 9, with $s=2$, $z=7=21_3$ and $\lambda_s(z)=1$ now. Here $5\cdot 7+A_2=39=1110_3$ and $10\cdot 7+A_2=74=2202_3$, so we get $t-1$ digit carries in the sum $5x+A_r=9(5y+A_{r-2})+(5z+A_2)$ and at least two more in the sum $10x+A_r=9(10y+A_{r-2})+(10z+A_2)$. Therefore $2(C-D)-\lambda_s(z)\geq 2(t+1-t)-1=1$.

\item[{\it Case} 13:] The last three digits of $x$ are $221$. We apply (\ref{induct3}) with $s=1$, so $z=1=1_3$ and $\lambda_s(z)=1$. There is only one digit carry in the sum $2x+A_r=3(2y+A_{r-1})+(2z+A_1)$, so $D=1$. The last two digits of $5y+A_{r-2}$ are 22, so we get at least two digit carries in the sum $5x+A_r=3(5y+A_{r-1})+(5z+A_1)$. Therefore $2(C-D)-\lambda_s(z)\geq 2(2-1)-1=1$.


\item[{\it Case} 14:] The last two digits of $x$ are $12$. Let $t$ be the number of consecutive 1's before the last digit. We apply (\ref{induct3}) with $s=1$, so $z=2=2_3$ and $\lambda_s(z)=0$. There are $t$ digit carries in the sum $x+A_r=3(y+A_{r-1})+(z+A_1)$ and no digit carry in $2x+A_r=3(2y+A_{r-1})+(2z+A_1)$, so $D=t$. On the other hand, the last $t$ digits of $5y+A_{r-1}$ and $10y+A_{r-1}$ are $\overbrace{22\ldots 22}^{t-1}0$ and $\overbrace{11\ldots 11}^{t-3}022$ respectively, and $5\cdot 2+A_1=11=102_3$ and $10\cdot 2+A_1=21=210_3$, so we get $t-1$ digit carries in the sum $5x+A_r=3(5y+A_{r-1})+(5z+A_1)$ and at least one more in the sum $10x+A_r=3(10y+A_{r-1})+(10z+A_1)$. Therefore $2(C-D)-\lambda_s(z)\geq 2(t-t)-0=0$.

\item[{\it Case} 15:] The last three digits of $x$ are 122. Let $t$ be the number of consecutive 1's before the 22. We apply (\ref{induct3}) with $s=2$, so $z=8=22_3$ and $\lambda_s(z)=-2$. There are $t$ digit carries in the sum $x+A_r=9(y+A_{r-2})+(z+A_2)$ and no digit carry in $2x+A_r=9(2y+A_{r-2})+(2z+A_2)$, so $D=t$. On the other hand, the last $t$ digits of $5y+A_{r-2}$ are $\overbrace{22\ldots 22}^{t-1}0$, and $5\cdot 8+A_2=44=1122_3$, so we get $t-1$ digit carries in the sum $5x+A_r=9(5y+A_{r-2})+(5z+A_2)$. Therefore $2(C-D)-\lambda_s(t)\geq 2(t-1-t)+2=0$.


\item[{\it Case} 16:] The last three digits of $x$ are 222. We apply (\ref{induct3}) with $s=1$, so $z=2=2_3$ and $\lambda_s(z)=0$. There is only one digit carry in the sum $2x+A_r=3(2y+A_{r-1})+(2z+A_1)$, so $D=1$. The last two digits of $5y+A_{r-2}$ are 22, so we get at least one digit carry in the sum $5x+A_r=3(5y+A_{r-1})+(5z+A_1)$. Therefore $2(C-D)-\lambda_s(z)\geq 2(1-1)=0$.
\end{enumerate}
\end{proof}

\begin{cor}
Let $r\geq 1$ and let $0< x <2^r-1$ an integer. Then
 $$
[5x+A_r]_r+[10x+A_r]_r\leq [x]_r+[x+A_r]_r+[2x+A_r]_r+6.
$$
\end{cor}

\begin{proof}
If $x=A_r=\frac{3^r-1}{2}$ or $r$ is even and $x=\frac{3^r-1}{4}$ or $x=\frac{3(3^r-1)}{4}$ then $[x]_r+[x+A_r]_r+[2x+A_r]_r=4r$ and the inequality is obvious. Otherwise, the $3$-adic expansion of $x$ contains two consecutive digits with are not $11$, $02$ or $20$. Multiplying $x$ by a suitable power of 3, we can assume that they are the last two digits.

Note that $x+A_r$ has at most $r+1$ digits, and if it has $r+1$ then the first one is $1$. In that case, $[x+A_r]_r=[x+A_r-3^r+1]=[x+A_r+1]-1$. Since the last two digits of $x$ are not $11$, the last two digits of $x+A_r$ are not $22$, so there is at most one digit carry in the sum $(x+A_r)+1$. Therefore $[x+A_r+1]-1\geq[x+A_r]-2$. In any case, we get $[x+A_r]_r\geq[x+A_r]-2$.

$2x+A_r$ has at most $r+1$ digits. If it has $r+1$, let $a\in\{1,2\}$ be the first one. Then $[2x+A_r]_r=[2x+A_r-a\cdot 3^r+a]=[2x+A_r+a]-a$. Since the last two digits of $x$ are not $02$ or $20$, the last two digits of $2x+A_r$ are not $21$ or $22$, so there is at most one digit carry in the sum $(2x+A_r)+a$. Therefore $[2x+A_r+a]-a\geq[2x+A_r]-2$. In any case, we get $[2x+A_r]_r\geq[2x+A_r]-2$.

So we have
$$
[5x+A_r]_r+[10x+A_r]_r\leq [5x+A_r]+[10x+A_r]\leq
$$
$$
\leq[x]+[x+A_r]+[2x+A_r]+2\leq[x]_r+[x+A_r]_r+2+[2x+A_r]_r+2+2
$$
\end{proof}

We conclude the proof of (\ref{V-main2}) by using the numerical Hasse-Davenport formula as in Theorem \ref{3 x 13}.
\end{proof}

\begin{thm}\label{28^times} In characteristic $p=3$, the lisse sheaf $\sH(\psi,28^\times)$ on $\G_m/\F_9$ has finite $G_{arith}$ and finite $G_{geom}$.
\end{thm}
\begin{proof}
By Lemma \ref{finitemonocritbis}, we must show that
$$V(28x) + V(2x) + V(-x) \ge V(4x) + V(14x),$$
for all  $x \in (\Q/\Z)_{\mathrm {prime\ to \ }p}$.  If $x\in\frac{1}{28}{\mathbb Z}$ we check it by hand, otherwise, just as in the proof of Theorem \ref{3 x 13}, it is equivalent to
$$
V(28x)+V(2x)\leq V(x)+V(4x)+V(14x)
$$
which, applying the duplication formula, is equivalent to
$$
V\left(14x+\frac{1}{2}\right)\leq V(x)+V\left(2x+\frac{1}{2}\right).
$$

In terms of the $[-]_{r}:=[-]_{3,r}$ function \cite[\S 4]{Ka-RL}, we need to show that, for all $r\geq 2$ and all integers $0<x<3^r-1$ we have
\begin{equation}\label{V-main3}
\left[14x+\frac{3^r-1}{2}\right]_r\leq [x]_r+\left[2x+\frac{3^r-1}{2}\right]_r.
\end{equation}

For a non-negative integer $x$, let $[x]$ denote the sum of the 3-adic digits of $x$.

\begin{lem}
 Let $r\geq 1$ and let $0\leq x <3^r$ an integer. Then
 $$
\left[14x+\frac{3^r-1}{2}\right]\leq [x]+\left[2x+\frac{3^r-1}{2}\right]+1.
$$
\end{lem}

\begin{proof}
 We proceed by induction on $r$: for $r\leq 3$ one checks it by computer. Let $r\geq 4$ and $0\leq x< 3^r$, and consider the $3$-adic expansion of $x$. By adding leading 0's as needed, we will assume that it has exactly $r$ digits. Let $A_r=\frac{3^r-1}{2}$.
 
 In all cases below we will follow this procedure: for some $s\leq r-2$, we write $x=3^sy+z$ with $y<3^{r-s}$, $z<3^s$. Let $C$ be the total number of digit carries in the sum $14x+A_r=3^s(14y+A_{r-s})+(14z+A_s)$ and $D$ the total number of digit carries in the sum  $2x+A_r=3^s(2y+A_{r-s})+(2z+A_s)$, and let $\lambda_s(z)=[14z+A_s]-[z]-[2z+A_s]$. If $2(C-D)-\lambda_s(z)\geq 0$, then
 \begin{equation}\label{induct5}
[14x+A_r]=[3^s(14y+A_{r-s})+(14z+A_s)]
=[14y+A_{r-s}]+[14z+A_s]-2C\leq
\end{equation}
$$
\leq [y]+[2y+A_{r-s}]+1+[z]+[2z+A_s]+\lambda_s(z)-2C\leq [y]+[2y+A_{r-s}]+1+[z]+[2z+A_s]-2D=
$$
$$
=[3^sy+z]+[3^s(2y+A_{r-s})+(2z+A_s)]+1=[x]+[2x+A_r]+1
$$
by induction. 

 \bigskip
 \begin{enumerate}
 \item[{\it Case} 1:] $x\equiv 0 \mod 3$. We apply (\ref{induct5}) with $s=1$, so $z=0$ and $C=D=\lambda_s(z)=0$.
 
\item[{\it Case} 2:] The last two digits of $x$ are $01$. We apply (\ref{induct5}) with $s=1$, so $z=1$, $\lambda_s(z)=1$ and $D=0$. Here $14z+A_1=15=120_3$, and the last digit of $14y+A_{r-1}$ is 1, so $C\geq 1$. Therefore $2(C-D)-\lambda_s(z)\geq 2-0-1=1$.

\item[{\it Case} 3:] The last three digits of $x$ are $011$ or $111$. We apply (\ref{induct5}) with $s=2$, so $z=4=11_3$ and $\lambda_s(z)=0$. Here $2z+A_2=12=110_3$ and the last digit of $2y+A_{r-2}$ is $1$ or $0$, so $D=0$. Therefore $2(C-D)-\lambda_s(z)=2C\geq 0$.

\item[{\it Case} 4:] The last four digits of $x$ are $0211$. We apply (\ref{induct5}) with $s=3$, so $z=22=211_3$ and $\lambda_s(z)=0$. Suppose that the last $t$ digits of $y$ are $2020\ldots 20$ (if $t$ is even) or $020\ldots 20$ (if $t$ is odd) and the previous one is not $0$ (if $t$ is even) or $2$ (if $t$ is odd). Then $2z+A_3=57=2010_3$ and the last $t$ digits of $2y+A_{r-3}$ are $\overbrace{22\ldots 22}^{t-1}1$ with the previous one (if it exists) different than 2, so $D=t$. On the other hand, $14z+A_3=321=102220_3$ and the last $t$ digits of $14y+A_{r-3}$ are $\overbrace{22\ldots 22}^{t-3}121$, so $C\geq t$ and therefore $2(C-D)-\lambda_s(t)\geq 2(t-t)=0$.

\item[{\it Case} 5:] The last four digits of $x$ are $1211$. We apply (\ref{induct5}) with $s=3$, so $z=22=211_3$ and $\lambda_s(z)=0$. Here $2z+A_3=57=2010_3$ and the last digit of $2y+A_{r-3}$ is $0$, so $D=0$. Therefore $2(C-D)-\lambda_s(z)=2C\geq 0$.

\item[{\it Case} 6:] The last four digits of $x$ are $2211$. We apply (\ref{induct5}) with $s=2$, so $z=4=11_3$ and $\lambda_s(z)=0$. Here $2z+A_2=12=110_3$ and the last two digits of $2y+A_{r-2}$ are $02$, so $D=1$. On the other hand, $14z+A_2=60=2020_3$ and the last two digits of $14y+A_{r-2}$ are $22$, so $C\geq 1$. Therefore $2(C-D)-\lambda_s(z)\geq 2(1-1)=0$.

\item[{\it Case} 7:] The last three digits of $x$ are $021$. Here we can proceed as in case 4 if we take $s=2$ and $z=7=21_3$ (so $\lambda_s(z)=-1$), since $2z+A_2=18=200_3$ and $14z+A_2=102=10210_3$.

\item[{\it Case} 8:] The last three digits of $x$ are $121$. We apply (\ref{induct5}) with $s=2$, so $z=7=21_3$ and $\lambda_s(z)=-1$. Since the last digit of $2y+A_{r-2}$ is 0, we have $D=0$, so $2(C-D)-\lambda_s(z)=2C+1>0$.

\item[{\it Case} 9:] The last three digits of $x$ are $221$. We apply (\ref{induct5}) with $s=1$, so $z=1=1_3$ and $\lambda_s(z)=1$. The last two digits of $2y+A_{r-1}$ are 02, so $D=1$. On the other hand, the last two digits of $14y+A_{r-1}$ are 22 and $14z+A_1=15=120_3$, so $C\geq 2$. Therefore $2(C-D)-\lambda_s(z)\geq 2(2-1)-1>0$.

\item[{\it Case} 10:] The last two digits of $x$ are $02$ or $12$, or the last three digits are $122$ or $222$. We apply (\ref{induct5}) with $s=1$, so $z=2=2_3$ and $\lambda_s(z)=-2$. Here $2z+A_1=5=12_3$ and the last two digits of $2y+A_{r-1}$ are not 22, so $D\leq 1$. Therefore $2(C-D)-\lambda_s(z)\geq 2(C-1)+2=2C\geq 0$.

\item[{\it Case} 11:] The last three digits of $x$ are 022. We apply (\ref{induct5}) with $s=2$, so $z=8=22_3$ and $\lambda_s(z)=-2$. Suppose as in case 4 that the last $t$ digits of $y$ are $2020\ldots 20$ (if $t$ is even) or $020\ldots 20$ (if $t$ is odd) and the previous one is not $0$ (if $t$ is even) or $2$ (if $t$ is odd). Then $2z+A_2=20=202_3$, so $D=t$. On the other hand, $14z+A_2=116=11022_3$, so $C\geq t-1$ and therefore $2(C-D)-\lambda_s(z)\geq 2(t-1-t)+2=0$.
\end{enumerate}
\end{proof}

\begin{cor}
Let $r\geq 1$ and let $0< x <3^r-1$ an integer. Then
 $$
[14x+A_r]_r\leq [x]_r+[2x+A_r]_r+3.
$$
\end{cor}

\begin{proof}
If $r$ is even and $x=\frac{3^r-1}{4}$ or $x=\frac{3(3^r-1)}{4}$ then $[2x+A_r]_r=2r$ and the inequality is obvious. Otherwise, the $3$-adic expansion of $x$ contains two consecutive digits with are not $02$ or $20$. Multiplying $x$ by a suitable power of 3, we can assume that they are the last two digits.

Note that $2x+A_r$ has at most $r+1$ digits. If it has $r+1$, let $a\in\{1,2\}$ be the first one. Then $[2x+A_r]_r=[2x+A_r-a\cdot 3^r+a]=[2x+A_r+a]-a$. Since the last two digits of $x$ are not $02$ or $20$, the last two digits of $2x+A_r$ are not $21$ or $22$, so there is at most one digit carry in the sum $(2x+A_r)+a$. Therefore $[2x+A_r+a]-a\geq[2x+A_r]-2$. In any case, we get $[2x+A_r]_r\geq[2x+A_r]-2$.

So we have
$$
[14x+A_r]_r\leq [14x+A_r]\leq
$$
$$
\leq[x]+[2x+A_r]+1\leq[x]_r+[2x+A_r]_r+3.
$$
\end{proof}

We conclude the proof of (\ref{V-main3})by using the numerical Hasse-Davenport formula as in Theorem \ref{3 x 13}.
\end{proof}

\section{Determination of some finite complex linear groups}

\begin{thm}\label{simple1}
Let $V = \CC^{12}$ and let $G < \sG :=\GL(V)$ be a finite irreducible subgroup. Suppose that all the following conditions hold:
\begin{enumerate}[\rm(i)]
\item $V$ is primitive and tensor indecomposable;
\item $G$ contains a subgroup $N$ of the form $N = C_3^5 \rtimes C_{11}$, with $C_{11}$ acting nontrivially on $Q:= \bfO_3(N) = C_3^5$.
\end{enumerate}
Then $G = \bfZ(G)H$ with $H \cong 6.\Suz$ in one of its two (up to equivalence) complex conjugate irreducible representations of degree $12$.
\end{thm}

\begin{proof}
(a) By the assumption, the $G$-module $V$ is irreducible, primitive and tensor indecomposable. Since $\dim(V) = 12$, it cannot be tensor induced. 
Hence, we can apply \cite[Proposition 2.8]{G-T} to obtain a finite subgroup $H < \SL(V)$ with $\bfZ(\sG)G=\bfZ(\sG)H$ which is almost quasisimple, that is, 
$S \lhd H/\bfZ(H) \leq \Aut(S)$ for some finite non-abelian simple group $S$. By \cite[Lemma 2.5]{G-T}, the layer $L=E(H)$ (which in this case is 
just the last term of the derived series of the complete inverse image of $S$ in $H$) is a finite quasisimple group acting irreducibly on $V$, whence
$\bfZ(L) \leq \bfZ(H)$ by Schur's Lemma.

Condition (ii) implies that the subgroup $C_{11}$ of $N$ acts irreducibly on $Q$ (considered as an $\F_3$-module), and so it acts 
fixed-point-freely on $Q \smallsetminus \{1\}$. In particular, $Q \cap \bfZ(\sG) = 1$.
By the construction of $H$ in the proof of \cite[Proposition 2.8]{G-T}, it contains the subgroup
$$Q_1 := \{ \al g \in \SL(V) \mid g \in Q, \al \in \C^\times \}$$ 
such that $\bfZ(\sG)Q = \bfZ(\sG)Q_1$. It follows that
$$Q_1/(Q_1 \cap \bfZ(H)) = Q_1/(Q_1 \cap \bfZ(\sG)) \cong Q/(Q \cap \bfZ(\sG)) \cong Q \cong C_3^5,$$
which implies that the almost simple group $H/\bfZ(H) \leq \Aut(S)$ has $3$-rank at least 5.

\smallskip
(ii) Applying the main result of \cite{H-M} to $L$, we now arrive at one of the following possibilities.

$\bullet$ $S = \Alt_{13}$, $\Alt_6$, $\SL_3(3)$, $\PSL_2(11)$, $\PSL_2(13)$, $\PSL_2(23)$, $\PSL_2(25)$, $\SU_3(4)$, $\PSp_4(5)$, $G_2(4)$, or  $M_{12}$. In all of these cases, the $3$-rank of $\Aut(S)$ is less than 5, see \cite{ATLAS}, a contradiction.

$\bullet$ $L = 6.\Suz$. In this case, since outer automorphisms of $L$ do not fix the isomorphism class of any complex irreducible representation 
of degree $12$ of $L$ (in fact, it fuses the two central elements of order $3$ of $L$ which act nontrivially on $V$), we see that $H/\bfC_H(L) \cong L/\bfZ(L)$, and so $H = \bfZ(H)L$ and $L= [L,L] = [H,H]$. As $\bfZ(\sG)G = \bfZ(\sG)H$, we conclude that $G = \bfZ(G)L$, as stated. 
\end{proof}
    
\begin{thm}\label{simple2}
Let $V = \CC^{24}$ and let $G < \GO(V)$ be a finite irreducible subgroup. Suppose that all the following conditions hold:
\begin{enumerate}[\rm(i)]
\item $V$ is primitive and tensor indecomposable;
\item $G$ contains a subgroup $N$ of the form $N = C_2^{11} \rtimes C_{23}$, with $C_{23}$ acting nontrivially on $Q := \bfO_2(N) \cong C_2^{11}$.
\end{enumerate}
Then $G \cong 2.\Co_1$ in its unique (up to equivalence) irreducible representation of degree $24$.
\end{thm}
  
\begin{proof}
(a) By the assumption, the $G$-module $V$ is irreducible, primitive and tensor indecomposable. Since $\dim(V) = 24$, it cannot be tensor induced. 
Hence, $G$ is almost quasisimple by \cite[Proposition 2.8]{G-T}, and so $S \lhd G/\bfZ(G) \leq \Aut(S)$ for some finite non-abelian simple group $S$. By \cite[Lemma 2.5]{G-T}, the layer $L=E(G)$ (which in this case again is 
the last term of the derived series of the complete inverse image of $S$ in $G$) is a finite quasisimple group acting irreducibly on $V$, whence
$\bfZ(L) \leq \bfZ(G) \leq C_2$ by Schur's Lemma.

Condition (ii) implies that the subgroup $C_{23}$ of $N$ acts irreducibly on $Q$ (considered as an $\F_2$-module), and so it acts 
fixed-point-freely on $Q \smallsetminus \{1\}$. In particular, $Q \cap \bfZ(G) = 1$.
It follows that
$$Q/(Q \cap \bfZ(G)) \cong Q \cong C_2^{11},$$
which implies that the almost simple group $G/\bfZ(G) \leq \Aut(S)$ has $2$-rank at least 11.

\smallskip
(ii) Applying the main result of \cite{H-M} to $L$, we now arrive at one of the following possibilities.

$\bullet$ $S = \Alt_7$, $\Alt_8$, $\PSL_3(4)$, $\SU_4(2)$, $\PSp_4(7)$, $\PSL_2(23)$, $\PSL_2(25)$, $\PSL_2(47)$, or 
$\PSL_2(49)$. In all of these cases, the $2$-rank of $\Aut(S)$ is less than 11, see \cite{ATLAS}, a contradiction.

$\bullet$ $L = \Alt_{25}$. Recalling that $\bfZ(G) \leq C_2$ and that $\Out(S) \cong C_2$ in this case, we see that $N/(N \cap \bfZ(G))$ contains a 
subgroup $C_{23} < S$ that acts nontrivially on $Q/(Q \cap \bfZ(G)) \cong C_2^{11}$. This implies that $S$ contains a subgroup
$N_1$ with $Q_1 := \bfO_2(N_1) \cong C_2^{11}$ and $N_1/Q_1 \cong C_{23}$ acting irreducibly on $Q_1$ (considered as an $\F_2$-module).
In turn, the latter implies that $C_{23}$ acts fixed-point-freely on $\Irr(Q_1) \smallsetminus \{1_{Q_1}\}$, and so any transitive permutation action of
$N_1$ with nontrivial $Q_1$-action must be on at least $1+23 = 24$ symbols. Now consider the natural action of $N_1 < S \cong \Alt_{25}$ on 
$25$ letters. This must admit at least one orbit $\Omega$ with nontrivial $Q_1$-action, and so $24 \leq |\Omega| \leq 25$ by the previous assertion.
But this is a contradiction, since neither $24$ nor $25$ divides $|N_1| = 2^{11} \cdot 23$.   

$\bullet$ $L = 2.\Co_1$. In this case, since $\Out(S) = 1$ (see \cite{ATLAS}), we conclude that $G = L$, as stated. 
\end{proof}

\section{Determination of the monodromy groups}
 \begin{thm}\label{2.Co_1}For the lisse sheaf $\sH(\psi,3 \times 13)$ on $\G_m/\F_4$, we have $G_{geom}=G_{arith} = 2.\Co_1$ in its unique 
 (up to equivalence)  $24$-dimensional irreducible representation (as the automorphism group of the Leech lattice).
 \end{thm}              
 \begin{proof}Choose (!) an embedding of $\overline{\Q}_{\ell}$ into $\C$. We will show that the result follows from Theorem \ref{simple2}.
 
 From Lemma \ref{det}, we have
 $$G_{geom} \lhd G_{arith} \subset \GO_{24}(\C).$$
 Because $\sH(\psi,3 \times 13)$ is geometrically irreducible, $G_{geom}$ (and a fortiori $G_{arith}$) is an irreducible subgroup of $\GO_{24}(\C)$. By Theorem \ref{3 x 13}, $G_{arith}$ (and a fortiori $G_{geom}$) is a finite subgroup. By Proposition \ref{indecompose} and Corollary \ref{pprim},  $G_{geom}$ (and a fortiori $G_{arith}$) is tensor indecomposable and primitive.
 
 By Lemma \ref{inertiaimage}, the image of the wild inertia group $P(\infty)$ is the Pontrayagin dual of the additive group of the field $\F_2(\mu_{23}) = \F_{2^{11}}$, acting as the direct sum of the $23$ characters $\sL_{\psi(23\zeta x)}$, indexed by $\zeta \in \mu_{23}$. The  group $I(\infty)/P(\infty)$ acts through its cyclic quotient $\mu_{23}$, with a primitive $23^{\mathrm {rd}}$ root of unity cyclically permuting the  $\sL_{\psi(23\zeta x)}$. Thus $G_{geom}$ (and a fortiori $G_{arith}$) contains the required $N = C_2^{11} \rtimes C_{23}$ subgroup.
 
 \end{proof}
 
 \begin{thm}\label{6.Suz}For each of the lisse sheaves $\sH(\psi,4 \times 5)$ and $\sH(\psi,28^\times)$ on $\G_m/\F_9$, we have
 $$G_{geom} =G_{arith} =6.\Suz$$
 in one of its two (up to equivalence) complex conjugate irreducible representations of degree $12$.
 \end{thm}
 \begin{proof}In this case, the result follows from Theorem \ref{simple1}. Just as in the proof of Theorem \ref{2.Co_1}, we see that both $G_{geom}$ and $G_{arith}$ are finite, irreducible, primitive,
 tensor indecomposable subgroups of $\GL_{12}(\C)$.
 
  By Lemma \ref{inertiaimage}, the image of the wild inertia group $P(\infty)$ is the Pontrayagin dual of the additive group of the field $\F_3(\mu_{11}) = \F_{3^{5}}$, acting as the direct sum of the $11$ characters $\sL_{\psi(11\zeta x)}$, indexed by $\zeta \in \mu_{11}$. The  group $I(\infty)/P(\infty)$ acts through its cyclic quotient $\mu_{11}$, with a primitive $11^{\mathrm {th}}$ root of unity cyclically permuting the  $\sL_{\psi(11\zeta x)}$. Thus $G_{geom}$ (and a fortiori $G_{arith}$) contains the required $N = C_3^{5} \rtimes C_{11}$ subgroup.
  
 Therefore by Theorem \ref{simple1}, the group $G_{arith}$ (and the group $G_{geom}$) is the group $6.\Suz$, augmented by some finite group of scalars. If $\beta$ is a scalar contained in
 $G_{arith}$, then $12\beta$ is its trace in the given $12$-dimensional representation of $G_{arith}$. But the traces of $G_{arith}$ lie in
 $\Q(\zeta_3)$, and thus $\beta$ lies in $\Q(\zeta_3)$. But $\beta$ is a root of unity, hence lies in $\mu_6$. But $\mu_6$ lies in 
$6.\Suz$, and thus $G_{arith}$ is $6.\Suz$, and a fortiori $G_{geom}$, which contains ``fewer" scalars, is also $6.\Suz$.
\end{proof}

\begin{rmk}
To see that $2.\Co_1$ actually contains $C_2^{11} \rtimes C_{23}$, note that $2.\Co_1$ contains $C_2^{12} \rtimes M_{24} > C_2^{12} \rtimes C_{23}$, see
\cite{ATLAS}. Next, as a $C_{23}$-module, $C_2^{12}$ is semisimple with a $1$-dimensional fixed point subspace, leading to the decomposition
$C_2^{12} \rtimes C_{23} = (C_2^{11} \rtimes C_{23}) \times C_2$. The same argument, using a maximal subgroup $C_3^5 \rtimes C_{11}$ of 
$\Suz$ \cite{ATLAS} shows that the full inverse image of this subgroup in $6.\Suz$ splits as $(C_3^5 \rtimes C_{11}) \times C_6$, and so 
$6.\Suz$ contains $C_3^5 \rtimes C_{11}$.
\end{rmk}

\section{Pullback to $\A^1$}
We begin by stating the simple (and well known) lemma that underlies the constructions of this section.
\begin{lem}\label{pullback}Let $\sH$ be a local system of $\G_m/\F_q$ whose local monodromy at $0$ is of finite order $M$ prime to $p$. For $N$  any  prime to $p$ multiple of  $M$, consider the pullback local system
 $$\sG(N):=[N]^\star \sH := [x \mapsto x^N]^\star \sH$$
 on $\G_m/\F_q$. Then we have the following results.
 \begin{enumerate}[\rm(i)]
 \item The local system $\sG(N)$ on $\G_m/\F_q$ has a unique extension to a local system on $\A^1/\F_q$, call it $\sG_0(N)$.
 \item The local systems $\sG(N)$ on $\G_m/\F_q$ and $\sG_0(N)$ on $\A^1/\F_q$ have the same $G_{arith}$ as each other, and
 the same $G_{geom}$ as each other.
 \item We have inclusions
 $$G_{arith,\sG(N)}  < G_{arith,\sH},\ \ G_{geom,\sG(N)} \lhd G_{geom,\sH},$$
 and  the quotient
 $$G_{geom,\sH}/G_{geom,\sG(N)}$$
 is a cyclic group of order dividing $N$.
 \end{enumerate}
\end{lem}
\begin{proof}(i) If such an extension exists, it must be $j_\star \sG(N)$ for $j:\G_m \subset \A^1$ the inclusion. This direct image is
lisse at $0$ precisely because the local monodromy of $ \sG(N)$ at $0$ is trivial. (ii) is simply the fact that $G_{geom}$ and $G_{arith}$ are birational invariants. (iii) is Galois theory, and the fact that the extension $\overline{\F_q}(x^{1/N})/\overline{\F_q}(x)$ is Galois, with cyclic Galois group $\mu_N(\overline{\F_q})$.
\end{proof}

\begin{thm}We have the following results.
\begin{enumerate}[\rm(i)]
\item The pullback local system $[39]^\star \sH(\psi,3 \times 13)$ on $\A^1/\F_4$ has $G_{geom}=G_{arith} = 2.\Co_1$.
\item The pullback local system $[20]^\star \sH(\psi,4 \times 5)$ on $\A^1/\F_9$ has $G_{geom}=G_{arith} = 6.\Suz$.
\item The pullback local system $[28]^\star \sH(\psi,28^\times)$ on $\A^1/\F_9$ has $G_{geom}=G_{arith} = 6.\Suz$.
\end{enumerate}
\end{thm}
 \begin{proof}For $G$ either of the groups $6.\Suz$ or $2.\Co_1$, $G$ is a perfect group, and hence contains no proper normal subgroup $H \lhd G$ for which $G/H$ is abelian. So in each case listed, it results from Lemma \ref{pullback}(iii) above that $G_{geom}$ remains unchanged, equal to $G$, when we pass from $\sH$ to its pullback $\sG(N)$. From the inclusion $G_{arith,\sG(N)}  < G_{arith,\sH}$, we have the a priori inclusion $G_{arith,\sG(N)}  < G$. Thus we have
 $$G=G_{geom,\sG(N)} \lhd G_{arith,\sG(N)} < G.$$
 \end{proof}         
 \begin{rmk} Although the hypergeometric sheaves in question are rigid local systems on $\G_m$, we do not see any reason their pullbacks to $\A^1$ need be rigid local systems on $\A^1$.
 \end{rmk}       
 
 \section{Appendix: Another approach to tensor indecomposability}
 
\begin{prop}\label{indecompose1}Let $V$ be a representation of $I$ which is the direct sum $T \oplus W$ of a nonzero tame representation $T$ (i.e., one on which $P$ acts trivially) and of an
irreducible representation $W$ which is totally wild (i.e., one in which $P$ has no nonzero invariants). Then we have
the following results.
\begin{enumerate}[\rm(i)]
\item If $\dim(V)$ is not a square,
then $V$ is linearly tensor indecomposable as a  representation of $I$: there do not exist representations $V_1$ and $V_2$ of $I$, each of dimension $\ge2$ and an isomorphism of representations $V_1\otimes V_2 \cong V$. 
\item If $\dim(V)$ is a square which is neither $4$ nor an even power of $p$, then $V$ is linearly tensor indecomposable as a  representation of $I$.
\item If $\dim(V)$ is an even power of $p$ and $\dim(T) \ge 2$,  then $V$ is linearly tensor indecomposable as a  representation of $I$.
\end{enumerate}
\end{prop}

\begin{proof}
We argue by contradiction, assuming we have $V_1\otimes V_2 \cong V$. Replacing each of $V_1, V_2, V$ by its semisimplification, we may further assume each is $I$-semisimple. As $P$ is normal in $I$ (or because the image of $P$ in
any continuous $\ell$-adic representation is finite), each of these representations is $P$-semisimple as well.

We have canonical decompositions $V_1$ and $V_2$ into direct sums
$$V_1 = T_1 \oplus W_1,\ \ V_2 =T_2 \oplus W_2,$$
where the $T_i$ are tame representations of $I$, and the $W_i$ are totally wild representations of $I$.

{\it Step 1.} All four of $T_1, T_2, W_1, W_2$ cannot be nonzero. If they were, then $V_1\otimes V_2$ would contain the sum of
$T_1\otimes W_2$ and $T_2 \otimes W_1$, each of which is a nonzero totally wild representation of $I$, contradicting that $W$ is irreducible.

{\it Step 2.} We cannot have $V_1 = T_1$. For then the wild part $W$ of $V$ is $T_1\otimes W_2$, so by irreducibility of  $W$ the dimension of $V_1= T_1$ is $1$.

{\it Step 3.} We cannot have $V_1=W_1$ and $T_2$ nonzero. In this case, we would have
$$T\oplus W \cong T_2\otimes W_1 \oplus W_1\otimes W_2.$$
From the irreduciblility of $W$, we see that $\dim(T_2)$ must be $1$, and that $W_1\otimes W_2$ must be entirely tame.

{\it Step 4.} Thus we must have $V_1 =W_1$ and $V_2 =W_2$. Write each of $W_1, W_2$ as a sum of $I$-irreducibles, say
$$W_1=\sum_i W_{1,i},\ \ W_2 =\sum_j W_{2,j}.$$
Then $V_1 \otimes V_2 =W_1 \otimes W_2$ is 
$$\sum_{i,j}W_{1,i}\otimes W_{2,j}.$$
Of these $\sum_{i,j}W_{1,i}\otimes W_{2,j}$, precisely one summand fails to be totally tame, for the wild part of  $V_1 \otimes V_2$, which is irreducible, is the sum of the wild parts of the $W_{1,i}\otimes W_{2,j}$. We then invoke the following lemma.

\begin{lem}\label{alltame}Let $W_1$ and $W_2$ be irreducible, totally wild representations of $I$. If $W_1\otimes W_2$ is entirely tame, then 
$\dim(W_1)=\dim(W_2) =1$ and $W_2^\vee \cong W_1\otimes (\rm{some\ tame\ character \ }\chi)$.
\end{lem}
\begin{proof}Decompose each of the $W_i$ into its $P$-isotypical components. By \cite[1.14.2]{Ka-GKM}, we know that each isotypical component is $P$-irreducible. Thus as $P$-representations, we have
$$W_1 = \sum_i N_i, \  \  W_2 = \sum_j M_j,$$
with the $N_i$ and the $M_j$ each  $P$-irreducible.
Then $W_1\otimes W_2$ is $\sum_{i,j}N_i\otimes M_j$.
In the tensor product $N_i\otimes M_j$ of two irreducible representations, the trivial representation occurs either not at all, or just once, and it occurs precisely when $M_j \cong  N_i^\vee$. To say that $W_1\otimes W_2$ is entirely tame is to say that each $N_i\otimes M_j$ is entirely trivial as $P$-representation, or in other words that each $N_i\otimes M_j$ is both one-dimensional and trivial. For this to
hold, each $N_i$ and each $M_j$ has dimension $1$, and $M_j \cong N_i^\vee$. for every pair $(i,j)$. Thus all the $M_j$ are isomorphic, each being $N_1^\vee$. Similarly, all the $N_i$ are isomorphic, each being $M_1^\vee$. But the various $P$-isotypical components of a given irreducible $W_i$ are pairwise nonisomorphic. Thus $W_1 =N_1$ and $W_2 =M_2$ are one-dimensional duals on $P$, so duals up to tensoring by a tame character on $I$.
\end{proof}

Returning to our situation
$$V_1 \otimes V_2 =\sum_{i,j}W_{1,i}\otimes W_{2,j},$$
we may renumber so that $W_{1,1}\otimes W_{2,1}$ is not totally tame, but all other $W_{1,i}\otimes W_{2,j}$ are totally tame.

Suppose now that $W_1$ is the sum of two or more irreducibles, then $V_1 \otimes V_2$ contains
$$(W_{1,1} +W_{1,2})\otimes W_{2,1},$$
and hence $W_{1,2}\otimes W_{2,1}$ is totally tame. By the above Lemma \ref{alltame}, $W_{2,1}$ is one dimensional, and
$$W_{2,1} \cong W_{1,2}^\vee \otimes (\rm{some\ tame\ character \ }\chi).$$
If $W_2$ is the sum of two or more irreducibles, then each product $W_{1,2})\otimes W_{2,j}$ must be totally tame, hence we have
$$W_{2,j} \cong W_{2,1}^\vee \otimes (\mbox{some tame character }\chi_j).$$
Thus $W_2$ is of the form
$$W_2 = (\rm{tame}\ \Tame_2, dim \ge 1)\otimes W_{2,1},$$
and $V_1 \otimes V_2$ contains
$$(W_{1,1} +W_{1,2})\otimes (\Tame_2 \otimes W_{2,1}).$$
In particular $V_1 \otimes V_2$ contains $\dim(\Tame_2)$ pieces of the form 
$$W_{1,1}\otimes W_{2,1}\otimes (\rm{some\ tame\ character}),$$
none of which is totally tame. Therefore $\Tame_2$ is one-dimensional, hence $W_2$ is one-dimensional, i.e., $V_2$ is
one-dimensional, contradiction.

Thus $W_1$ is a single irreducible. Repeating the argument with $W_1$ and $W_2$ interchanged, $W_2$ must also be
a single irreducible. If $W_1 \otimes W_2$ has a nonzero tame part, say contains a tame character $\chi$,
then $W_1 \otimes W_2\otimes \overline{ \chi}$ contains $\triv$, and hence 
$$W_2 \cong W_1^\vee \otimes \chi.$$
But $W_1 \otimes W_2$ also has a nonzero (in fact irreducible) wild part, hence $\dim(W_1) \ge 2$ (otherwise $W_1 \otimes W_2$ will be $\chi$ alone). Thus $\dim(V_1) =\dim(V_2) = \dim(W_1)$, and $\dim(V)$ is a square.

We now examine the situation in which $\dim(V)$ is a square $n^2$. Thus 
$$V \cong W_1\otimes W_2 =\End(W_1)\otimes \chi,$$
i.e., $$V \otimes \overline{ \chi} \cong \End(W_1).$$
Now $V \otimes \overline{ \chi}$ is itself the sum of a nonzero tame part and an irreducible totally wild part, and $\dim(W_1)=n$. So the question becomes, when is it possible that for a $W$ of dimension $n$, $\End(W)$ is the sum of a nonzero tame part and an irreducible totally wild part. Let us refer to this as ``the $\End$ situation". This is the situation we would like to rule out.

We first show that if $n$ is prime to $p$, the $\End$ situation can only arise when $n=2$.
Denote by $I(n) \lhd I$ the unique open subgroup of index $n$. Thus $I/I(n) \cong \mu_n$.
Then $W$ is the sum of $n$ $P$-isotypical components $N_i$, each of which is one dimensional, stable by $I(n)$, and each of which is $P$-inequivalent to any of its nontrivial multiplicative translates $\mbox{MultTransl}_\zeta(N_i)$ by nontrivial $n$'th roots of unity $\zeta$. If we fix one of them, say $N :=N_1$, then as $P$-representation 
$$W \cong \oplus_{\zeta \in\mu_n}\mbox{MultTransl}_\zeta(N),$$
and hence as $P$-representation
$$\End(W) \cong \bigoplus_{(\zeta_1,\zeta_2) \in \mu_n \times \mu_n} \mbox{MultTransl}_{\zeta_1}(N)\otimes \mbox{MultTransl}_{\zeta_2}(N^\vee).$$
Each of these $n^2$ pieces is $I(n)$-stable. The $n$ ``diagonal" summands 
$$ \mbox{MultTransl}_{\zeta}(N)\otimes \mbox{MultTransl}_{\zeta}(N^\vee)$$
are $P$-trivial, and their sum is the tame part of $\End(W_1)$. 
The remaining $n(n-1)$ summands can be put together into $n-1$ pieces, as follows.
Start with the $n-1$ summands
$$N \otimes \mbox{MultTransl}_{\zeta_1}(N^\vee), \zeta_1 \neq 1.$$
For each, form the sum
$$\bigoplus_{\zeta_2}\mbox{MultTransl}_{\zeta_2}(N \otimes \mbox{MultTransl}_{\zeta_1}(N^\vee)).$$
Each of these $n-1$ sums is $I$-stable and totally wild. [It is the induction from $I(n)$ to $I$ of  
$N \otimes \mbox{MultTransl}_{\zeta_1}(N^\vee)$.] Thus we have at least $n-1$ totally wild constituents in $V \otimes \overline{ \chi}$.
But its wild part is irreducible, which is only possible if $n-1=1$, i.e., if $n=2$.

We next show that if $n=n_0q$ with $n_0$ prime to $p$ and $q$ a strictly positive power of $p$, the $\End$ situation can only arise if $n_0=1$. We argue by contradiction. Suppose, then, that $n_0 > 1$.
Denote by $I(n_0) \lhd I$ the unique open subgroup of index $n_0$. Thus $I/I(n_0) \cong \mu_{n_0}$.
Then $W$ is the sum of $n_0$ $P$-isotypical components $N_i$, each of which is $q$- dimensional, $P$-irreducible, stable by $I(n_0)$, and each of which is $P$-inequivalent to any of its nontrivial multiplicative translates $\mbox{MultTransl}_\zeta(N_i)$ by nontrivial $n_0$'th roots of unity $\zeta$. If we fix one of them, say $N :=N_1$, then as $P$-representation 
$$W \cong \oplus_{\zeta \in\mu_{n_0}}\mbox{MultTransl}_\zeta(N),$$
and hence as $P$-representation
$$\End(W) \cong \bigoplus_{(\zeta_1,\zeta_2) \in \mu_{n_0} \times \mu_{n_0}} \mbox{MultTransl}_{\zeta_1}(N)\otimes \mbox{MultTransl}_{\zeta_2}(N^\vee).$$
Each of these $n_0^2$ pieces is $I(n_0)$-stable. 

If $\zeta_1=\zeta_2$, the piece
$$\mbox{MultTransl}_{\zeta_1}(N)\otimes \mbox{MultTransl}_{\zeta_2}(N^\vee) =$$
$$=\mbox{MultTransl}_{\zeta_1}(N\otimes N^\vee)$$
is the direct sum of a single tame character with a totally wild part of dimension $q^2-1$ (simply because $N$ is $P$-irreducible). If If $\zeta_1\neq \zeta_2$, the piece
$$\mbox{MultTransl}_{\zeta_1}(N)\otimes \mbox{MultTransl}_{\zeta_2}(N^\vee) $$
is totally wild, of dimension $q^2$. Assembling these $n_0^2$ pieces into $n_0$ $I$-stable pieces as in the discussion of the prime to $p$ case, we get $n_0$ $I$-stable summands, each of which has a nonzero totally wild piece. But the totally wild part of $\End(W)$ is irreducible, contradiction.

Now we analyze the $\End$ situation when $n=q$ is a strictly positive power of $p$. Thus $W$ is $I$ irreducible of rank $q$. By \cite[1.14.2]{Ka-GKM}, $W$ is $P$-irreducible. Therefore the space $\End(W)^P$ of $P$-invariants in $\End(W)$ is one-dimensional,
which is to say that $\End(W)$ is the sum of a one-dimensional tame part and a totally wild part of dimension $q^2-1$.

\end{proof}

\begin{cor}\label{indecompose 2}Let $V$ be a representation of $I$ which is the direct sum $T \oplus W$ of a nonzero tame representation $T$ (i.e., one on which $P$ acts trivially) and of an
irreducible representation $W$ which is totally wild (i.e., one in which $P$ has no nonzero invariants). Suppose that $\dim(V)$ is neither $4$ nor an even power of $p$, or that $\dim(V)$ is an even power of $p$ but that its tame part $T$ has dimension $\ge 2$.  Suppose further that the representation of $I$ on $V$ factors through a finite quotient $\Gamma$ of I. Then $I$, or equivalently $\Gamma$, stabilizes no decomposition $V = A \otimes B$ with $\dim(A), \dim(B) > 1$.
\end{cor}
\begin{proof}Suppose that  $\Gamma$ stabilizes such a decomposition $V = A \otimes B$. Then each $\gamma \in \Gamma$ acts on $A \otimes B$ as $X_\gamma \otimes Y_\gamma$, for some non-unique $X_\gamma \in \GL(A), Y_\gamma \in \GL(B)$. However, each of $X_\gamma, Y_\gamma$ is unique up to multiplication by an invertible scalar. To see this notice
that if
$$X_\gamma \otimes Y_\gamma=X^{'}_\gamma \otimes Y^{'}_\gamma,$$
then
$$(X^{'}_\gamma)^{-1}X_\gamma \otimes Y_\gamma (Y^{'}_\gamma)^{-1}$$
is the identity in $\GL(A\otimes B)$, which in turn implies that each of $(X^{'}_\gamma)^{-1}X_\gamma$ and $Y_\gamma (Y^{'}_\gamma)^{-1}$ is an invertible scalar, in $ \GL(A)$ and  $\GL(B)$ respectively.
Thus both $A$ and $B$ are projective representations of $\Gamma$. View these projective representations as projective representations of $I$. Then the cocycles defining them are locally constant functions on $I \times I$. The obstructions to linearizing them lie in
$H^2(I,\mu_{\dim(A)})$ and in $H^2(I,\mu_{\dim(B)})$ respectively. Both of these groups vanish, because $I$ has
cohomological dimension $\le 1$, cf. \cite[Chapter II, 3.3, c)]{Serre-CG}. Let us choose linear representations $\A$ and $\B$ which lift the projective representations $A$ and $B$. Then $\A \otimes \B$ and $V$ are equivalent as projective representations, hence they ``differ" by some linear character $\rho$ of $I$. 

Replacing $\B$ by $\B\otimes \rho^{-1}$,
we have an isomorphism $V \cong \A\otimes \B$ of $I$-representations. Now apply Proposition \ref{indecompose1} to see that this is impossible unless $\dim(V)$ is either $4$ or an even power of $p$.
\end{proof}

\begin{rmk}\label{whyfinite}In the absence of the hypothesis in Corollary \ref{indecompose 2} that the representation on $V$ factors through a finite quotient of $I$, we do not see how to guarantee that $A$ and $B$ as projective representations of $I$ are defined by locally constant cocycles. This local constance is what allows us to apply the fact that the profinite group $I$ has
cohomological dimension $\le 1$.
\end{rmk}

\end{document}